\documentclass{amsproc}
\usepackage{amsmath} 
\usepackage{amssymb} 
\usepackage{amsthm} 
\usepackage{amsfonts} 
\usepackage{tikz-cd} 
\usepackage{enumitem} 
\usepackage{hyperref}
\usepackage[mathscr]{eucal}
\hypersetup{colorlinks}
\usepackage{cleveref} 

\newtheorem{THM}{Theorem}
\newtheorem{COR}[THM]{Corollary}
\newtheorem{thm}{Theorem}
\numberwithin{thm}{section}
\newtheorem{lem}[thm]{Lemma}

\newtheorem{prop}[thm]{Proposition}
\newtheorem{cor}[thm]{Corollary}
\theoremstyle{definition}
\newtheorem{question}[THM]{Question}

\DeclareMathOperator{\out}{Out}
\DeclareMathOperator{\aut}{Aut}
\DeclareMathOperator{\inn}{Inn}
\DeclareMathOperator{\hol}{Hol}

\DeclareMathOperator{\st}{st}
\DeclareMathOperator{\en}{en}

\newcommand{\doublebackslash}{\backslash\mkern-5mu\backslash}

\usepackage{import}

\begin{document}
\title[One-endedness of outer automorphism groups of free products]{One-endedness of outer automorphism groups
of free products of finite and cyclic groups}
\author{Rylee Alanza Lyman}
\address{Smith Hall Room 216, 101 Warren St Newark. NJ USA 07103}
\email{rylee.lyman@rutgers.edu}
\begin{abstract}
  The main result of this paper is that the outer automorphism group
  of a free product of finite groups and cyclic groups
  is semistable at infinity (provided it is one ended)
  or semistable at each end.
  In a previous paper, we showed that the group of outer automorphisms
  of the free product of two nontrivial finite groups with an infinite cyclic group
  has infinitely many ends, despite being of virtual cohomological dimension two.
  We also prove that aside from this exception,
  having virtual cohomological dimension at least two implies
  the outer automorphism group of a free product of finite and cyclic groups
  is one ended. As a corollary, the outer automorphism groups
  of the free product of four finite groups or the free product of a single
  finite group with a free group of rank two 
  are virtual duality groups of dimension two,
  in contrast with the above example.
  Our proof is inspired by methods of Vogtmann,
  applied to a complex first studied in another guise by Krsti\'c and Vogtmann.
\end{abstract}
\maketitle
\paragraph{Duality groups.}
A group $G$ of type $FP_\infty$ is a \emph{duality group of dimension $n$}
in the sense of Bieri--Eckmann~\cite{BieriEckmann}
if there exists a $G$-module $D$
such that for any $G$-module $A$
and any integer $k$,
we have an isomorphism
\[
  H^k(G;A) \cong H_{n-k}(G;D\otimes_\mathbb{Z}A),
\]
where $G$ acts on $D\otimes_\mathbb{Z} A$ diagonally.
The dimension, $n$, turns out to be equal to the cohomological dimension of $G$,
which is thus torsion free,
and is therefore determined by $G$.
Similarly, the dualizing module $D$
turns out to be determined by $G$;
it is isomorphic to $H^n(G; \mathbb{Z}G)$.

Many important examples of groups are duality groups,
or more generally, \emph{virtual duality groups,}
having a duality group as a subgroup of finite index.
(Virtual duality groups are thus allowed to have torsion,
but they will be virtually torsion free.)
Among these are, of course,
groups having a closed manifold as a $K(\pi,1)$,
but also fundamental groups of knot complements,
mapping class groups of closed surfaces,
and, most relevant to the current paper,
$\out(F_n)$~\cite{BestvinaFeighnDuality}.

This paper is motivated by the following question.

\begin{question}\label{mainquestion}
  Let $F$ be a free product of finite and cyclic groups.
  Is the outer automorphism group of $F$ a virtual duality group?
\end{question}

The group $\out(F)$ is of type $FP_\infty$
and virtually torsion free, so the question makes sense.
As we showed, but did not remark, in~\cite{MyselfCAT0},
the answer to this question is sometimes ``no''.

We prove the following theorem.

\begin{THM}\label{maintheorem}
  Let $F = A_1* \cdots *A_n * F_k$
  be a free product of the nontrivial finite groups $A_i$ with a free group of rank $k$.
  The invariants $n$ and $k$ determine the number of ends of
  $\out(F)$ in the following way.
  \[
    \begin{cases}
      n \le 1,\ k\le 1\text{ or } (n,k) = (2,0)
      & \out(F) \text{ is finite, i.e.\ has zero ends.} \\
      (n,k) = (3,0),\ (2,1),\ (0,2)
      & \out(F) \text{ has infinitely many ends.} \\
      \text{otherwise} & \out(F) \text{ has one end.}
    \end{cases}
  \]
\end{THM}

\paragraph{Ends of groups.}
Let us explain.
Celebrated results of Freudenthal and Hopf
say that a finitely generated group $G$
has $0$, $1$, $2$, or infinitely many \emph{ends.}
An \emph{end} of a topological space $X$
is a compatible choice,
for every compact subset $K$ of $X$,
of a path component of the complement $X - K$.
Put another way,
if $K \subset K'$ are compact subsets of $X$,
every path component of $X - K'$ determines a path component of $X - K'$,
and in fact we have an inverse system of sets
(or spaces, if we'd like to give them the discrete topology).
The \emph{space of ends} of $X$,
is the inverse limit
\[
  \varprojlim_K \pi_0(X - K)
\]
and the \emph{number} of ends is its cardinality.
What Freudenthal and Hopf prove
is that the number of ends of a connected CW-complex, 
or if you'd like, a Cayley graph,
that a finitely generated group $G$ acts on
with finite stabilizers and finite (i.e.\ compact) quotient
is in fact an invariant of the group,
and that if a group has at least three ends, it has infinitely many.

One way to think about one-endedness, which is in some sense generic,
is to say that, supposing we have fixed an exhaustion of $X$
by compact sets $K_0 \subset K_1 \subset \cdots$
with the property that $K_n$ is contained in the interior of $K_{n+1}$,
for any two points ``sufficiently far from $K_n$'',
say belonging to $X - K_N$ for $N$ sufficiently larger than $n$,
those two points
can be connected by a path that avoids $K_n$.
Thus for example the line is not one ended
(in fact it has two ends),
while the plane is.

\paragraph{End invariants of groups.}
This reframes one-endedness as \emph{$0$-connectivity at infinity,}
and naturally leads one to generalize to questions about
higher homotopy or homology groups ``at infinity''.
The plane, then, would fail to be simply connected at infinity,
since a loop encircling our compact set $K_N$
is not nullhomotopic in $X - K_n$ as soon as $K_n$ is nonempty,
while $\mathbb{R}^n$ for larger $n$ \emph{would} be simply connected at infinity.

To ask these questions for groups,
one needs to be slightly careful.
After all, every finitely generated group has a Cayley graph,
and graphs have no higher homotopy groups.
In the case of the plane,
what is interesting is that every loop \emph{is} nullhomotopic
in the plane as a whole,
but may \emph{not} remain so after removing a compact set.
Thus we put restrictions on the spaces in question,
asking them to be $n$-connected in the ordinary sense,
before we ask whether they are $n$-connected at infinity.
This has the effect for groups of restricting our attention
to groups with certain \emph{finiteness properties,}
which generalize being finitely generated or finitely presented.

\paragraph{Semistability.}
Even with this restriction, there is a subtlety:
the fundamental group---for instance---needs a basepoint,
but as we work in the complements of a compact exhaustion
no single basepoint will do.
Making a choice of basepoint for $\pi_1(X-K_n, \star_n)$ for each $n$
and a path from $\star_n$ to $\star_{n-1}$,
we produce a \emph{proper ray} in $X$.
To show that resulting inverse limit 
\[
  \varprojlim_n \pi_1(X-K_n,\star_n)
\]
is independent of our choices,
we would need to show that any two proper rays
(ending in the same end if there is more than one)
are properly homotopic.
For locally finite complexes $X$,
so in particular for finitely presented groups,
this is equivalent by~\cite[Proposition 16.1.2]{Geoghegan} to the condition that
this inverse system of groups is \emph{semistable}
or satisfies the \emph{Mittag-Leffler condition:}
for each $n$ there exists $N$ such that for $k \ge N$,
the image of $\pi_1(X - K_N,\star_n)$ in $\pi_1(X - K_n,\star_n)$
is equal to the image of $\pi_1(X - K_k, \star_k)$ in $\pi_1(X- K_n,\star_n)$.

Unwinding this definition, we see that it is equivalent to the claim that
every loop in $X - K_N$ based at $\star_N$ 
may be pushed off to $X - K_k$ (along our proper ray)
by a homotopy that avoids $K_n$.

The relevance to group cohomology is that if
$G$ is \emph{semistable at each end} in the above sense, then $H^2(G,\mathbb{Z}G)$
is a free abelian group~\cite[Theorem 16.5.1]{Geoghegan}.
It turns out that if $G$ is $(n-2)$-connected at infinity,
(as the plane and hence $\mathbb{Z}^2$ is $0$-connected at infinity, for instance,
although merely being $(n-2)$-acyclic will do)
then $H^i(G, \mathbb{Z}G) = 0$ for $i < n$,
and $H^n(G, \mathbb{Z}G)$ is torsion free~\cite{Geoghegan,BradyMeier}.
(Thus for a one-ended group semistability is somewhere between
one-endedness and simple connectivity at infinity.)
A group $G$ of type $FP_\infty$
is a duality group of dimension $n$
if and only if $H^i(G, \mathbb{Z}G) = 0$
for $i \ne n$
and $H^n(G,\mathbb{Z}G)$
is nontrivial and torsion free.
If additionally $G$ has cohomological dimension $n$,
then $H^i(G, \mathbb{Z}G) = 0$ for $i > n$
and $H^n(G, \mathbb{Z}G)$ is nontrivial.

Thus the combination of $(n-2)$-connectivity at infinity
with having virtual cohomological dimension $n$
implies that a group $G$ of type $FP_\infty$ is a virtual duality group.

Additionally, let us remark that it is an open question
whether every finitely presented group $G$
satisfies that $H^2(G,\mathbb{Z}G)$ is free abelian,
hence whether every finitely presented group is semistable at each end.
We are able to answer this in the affirmative for $\out(F)$
when $F$ is a free product of finite and cyclic groups.
The following is our main result.

\begin{THM}\label{semistability}
  Let $F = A_1 * \cdots * A_n * F_k$ be a free product of the finite groups
  $A_i$ with a free group of rank $k$.
  The group $\out(F)$ of outer automorphisms of $F$
  is semistable at each end.
\end{THM}

\begin{COR}
  For $F$ as above, writing $G = \out(F)$,
  we have that $H^2(G,\mathbb{Z}G)$ is free abelian.
\end{COR}

\paragraph{The case of $\out(F)$.}
From here we discuss $\out(F)$,
where $F = A_1 * \cdots * A_n * F_k$
is the free product of the finite groups $A_i$
with a free group of rank $k$,
and $\out(F)$ denotes its \emph{outer automorphism group,}
$\aut(F) / \inn(F)$.
The virtual cohomological dimension of $\out(F)$
was computed by Krsti\'c and Vogtmann~\cite{KrsticVogtmann};
it is equal to the dimension of a certain contractible simplicial complex $L$
that $\out(F)$ acts on with finite stabilizers and finite quotient.
We will call this complex the \emph{spine of (reduced) Outer Space} for $F$.
The groups $\out(F)$,
which are virtually torsion free,
are thus virtually of type $F$,
hence certainly of type $FP_\infty$.
In any case, the dimension of $L$ can be computed from the invariants $n$ and $k$
by the following rule
\[
  \dim(L) =
  \begin{cases}
    2k + n - 2 & n \ge 2, \\
    \max\{2k + n - 3, 0\} & \text{otherwise}.
  \end{cases}
\]
Thus, to prove that $\out(F)$ is a virtual duality group,
it suffices to show that it is $(2k + n - 4)$-connected at infinity
if $n \ge 2$ or $(2k-n-5)$-connected at infinity otherwise.
In particular, when $L$ has dimension $2$,
we need only show that $\out(F)$ is one ended.
There are three such cases:
$F = A_1 * A_2 * A_3 * A_4$,
$F = A_1 * A_2 * \mathbb{Z}$
and $F = A_1 * F_2$.
In the first case, one-endedness
follows from a result of Das~\cite{Das},
who shows that $\out(F)$ is \emph{thick} when $F = A_1 * \cdots * A_n$
and $n \ge 4$.
The third case is covered by \Cref{maintheorem},
but in the second,
\Cref{maintheorem}
states that $\out(F)$ has infinitely many ends:
this is part of the main result of~\cite{MyselfCAT0}.
This is why the answer to \Cref{mainquestion} is ``no'' for some choices of $F$.
We view \Cref{maintheorem} as partial progress 
towards a full answer to \Cref{mainquestion}.

In view of \Cref{semistability}
and the fact that $\out(A_1 * A_2 * A_3 * A_4)$
and $\out(A_1 * F_2)$ are one-ended groups of virtual cohomological dimension two,
these groups have a well-defined \emph{fundamental group at infinity,}
namely the inverse limit of the fundamental groups as discussed
in the semistability paragraph above.
Because $H^2(G,\mathbb{Z}G)$ is nontrivial,
this fundamental group is nontrivial.
What is it?

\begin{question}
  What is the fundamental group of $\out(A_1 * A_2 * A_3 * A_4)$
  or $\out(A_1 * F_2)$ at infinity?
\end{question}

Our proof of \Cref{maintheorem}
is inspired by Vogtmann's paper~\cite{Vogtmann}.
There she analyzes the combinatorial structure
of links of certain vertices of the Spine of Outer Space (for $F = F_k$)
to push paths and homotopies off to infinity.
The Spine of Outer Space for general $F$
likewise has a rich combinatorial structure;
the present paper is a first contribution towards an understanding of it in general.

Like the papers~\cite{CullerVogtmann,KrsticVogtmann},
we proceed by putting a kind of ``height function''
on certain vertices of the complex $L$
of minimal complexity
we term ``briar patches''.
(A briar patch may have petals, like a rose,
and thorns like a thistle, but may also sprawl somewhat.)
Vertices of $L$ correspond to marked graphs of groups
(or if you'd rather, actions of $F$ on trees),
and the height function records the lengths of immersed loops 
associated to a predetermined finite set of conjugacy classes
(or if you'd rather, the hyperbolic translation lengths of these conjugacy classes).

The proof of contractibility of $L$ in~\cite{CullerVogtmann,KrsticVogtmann}
goes by showing, using ``peak-reduction'' techniques,
that these height functions behave somewhat like
non-singular Morse functions,
allowing one to contract the whole complex
onto the subcomplex where the function is minimized,
and then arguing that by carefully choosing the finite set of conjugacy classes,
this minimal subcomplex is manifestly contractible.

Like Vogtmann in~\cite{Vogtmann},
we'd like to do the opposite:
The height function provides an exhaustion of $L$;
for well-chosen finite sets, this is a compact exhaustion,
defining for us a sequence of compact subsets of $L$ we call ``balls''.
Analyzing the combinatorics of the spine
allows us to understand how our height function changes
as we move around the star of a single vertex of $L$.

Firstly, we restrict our attention to a family of ``standard paths'',
which look like a sequence of one-edge expand--collapse
moves (the term is ``Whitehead move'') between briar patches.
In the case $F = F_k$,
these moves between roses have the effect
of applying Whitehead automorphisms to the fundamental group, hence the name.
Supposing that $\tau_{2k}$ is a briar patch along such a path which is a kind of
local minimum for the height function,
the game is to find a new standard path
between the adjacent briar patches $\tau_{2(k-1)}$ and $\tau_{2(k+1)}$
that passes through briar patches of height strictly greater than that of $\tau$.
By doing so while staying within the star of $\tau$,
we have a homotopy between the old path and the new one,
and by iteratively applying this method to all local minima,
we are able to push our path off to infinity.
It turns out that when this strategy of proof works,
it suffices to prove semistability as well;
we turn to this in \Cref{semistabilitysection}.

Let us remark that our proof of \Cref{maintheorem}
is tailored to its application in \Cref{semistability},
for which we really need to be able to push paths off to infinity.
For a reader interested in an ``optimal'' proof of \Cref{maintheorem},
we give a sketch of a simpler argument suggested to us by the anonymous referee
in an appendix.

\paragraph{Acknowledgments.}
The author is pleased to thank Kim Ruane for suggesting
that \Cref{semistability} might be in reach given \Cref{maintheorem},
the anonymous referee for the argument sketched in the appendix,
and Mike Mihalik for helpful comments on a preliminary version.
This material is based on work supported by the National Science Foundation
under Award No. DMS-2202942.
\section{The Complex \texorpdfstring{$L$}{L}}\label{introsection}
We work in the \emph{spine of reduced Outer Space for $F$,}
a simplicial complex which we denote as $L = L(F)$.
A more detailed definition is given in~\cite{MyselfCAT0}.
A vertex of $L$ corresponds to an action of $F$ on a simplicial tree $T$ with finite stabilizers
(and in fact trivial edge stabilizers).
In this paper it will be more convenient 
to work in the quotient graph of groups $\mathcal{G} = F\doublebackslash T$.
We assume a certain amount of familiarity with graphs of groups in this paper
and refer the reader to~\cite{Trees},~\cite{Bass},~\cite{ScottWall} 
and~\cite[Section 1]{MyTrainTracks}
for background on graphs of groups.
We will follow the notation in~\cite{MyTrainTracks}, 
although almost everything we write is standard.
The main innovation of~\cite{MyTrainTracks}
that we need is the notion of a \emph{homotopy equivalence}
between graphs of groups
and homotopy of maps of graphs of groups.
For us the precise definition is less important than the following consequences:
if $\sigma\colon \mathbb{G} \to \mathcal{G}$ is a homotopy equivalence,
it determines (up to inner automorphism) an isomorphism
$\sigma_\sharp\colon \pi_1(\mathbb{G}) \cong \pi_1(\mathcal{G})$
and under this isomorphism the corresponding Bass--Serre trees $S$ and $T$
have the same subgroups which are \emph{elliptic,} that is, fixing a point of the tree.
For a precise definition, see~\cite{MyTrainTracks}.

Fix a finite graph of finite groups $\mathbb{G}$,
a basepoint $\star \in \mathbb{G}$
and an identification $F\cong \pi_1(\mathbb{G},\star)$.
A \emph{marked graph of groups} $\tau = (\mathcal{G},\sigma)$ 
is a graph of finite groups $\mathcal{G}$ 
together with a homotopy equivalence $\sigma\colon \mathbb{G} \to \mathcal{G}$
called the \emph{marking.}
Two marked graphs of groups $(\mathcal{G},\sigma)$ and $(\mathcal{G}',\sigma')$
are \emph{equivalent} if there is an isomorphism $h\colon \mathcal{G} \to \mathcal{G}'$ 
such that the following diagram commutes up to homotopy
\[  \begin{tikzcd}[row sep=tiny]
    & \mathcal{G} \ar[dd,"h"] \\
    \mathbb{G} \ar[ur, "\sigma"] \ar[dr, "\sigma'"'] & \\
    & \mathcal{G}'.
\end{tikzcd} \]

\paragraph{The complex $L$.}
A vertex of the complex $L$ is determined by
an equivalence class of marked graphs of groups $(\mathcal{G},\sigma)$
satisfying the following conditions.
\begin{enumerate}
    \item Edge groups of $\mathcal{G}$ are trivial.
    \item Each valence-one and valence-two vertex of $\mathcal{G}$ 
      has nontrivial vertex group.
    \item If an (open) edge $e$ separates $\mathcal{G}$,
      each component of the complement 
      contains a vertex with nontrivial vertex group.
\end{enumerate}
A marked graph of groups $\mathcal{G}$ determining a vertex of $L$,
is \emph{reduced} if collapsing any edge of $\mathcal{G}$
yields a marked graph of groups not homotopy equivalent to $\mathcal{G}$
(because the collapse map cannot be ``undone'',
for example because maybe 
the resulting graph of groups has an infinite vertex group).
The third condition above
is equivalent to asking that for each edge $e$ of $\mathcal{G}$,
there is a reduced marked graph of groups $\mathcal{G}'$
which may be obtained from $\mathcal{G}$ by collapsing edges
in which the edge $e$ is not collapsed.
In the language of~\cite{Clay,GuirardelLevittDeformation},
we say the edge $e$ is \emph{surviving.}

Two vertices $(\mathcal{G},\sigma)$ and $(\mathcal{G}',\sigma')$
are connected by an oriented edge 
if a marked graph of groups equivalent to $(\mathcal{G}',\sigma')$
can be obtained from $(\mathcal{G},\sigma)$ 
by collapsing certain edges in $\mathcal{G}$,
and in general $L$ possesses a simplex whenever its $1$-skeleton is present.
As observed in~\cite[Theorem 3.1]{Myself},
the main result of Krsti\'c--Vogtmann~\cite{KrsticVogtmann}
implies that $L$ is contractible.
The dimension of $L$ is $2k + n - 2$ 
when $n \ge 2$ and $2k + n - 3$ when $n \le 1$ and $k > 1$.

We assume, in view of~\cite{Vogtmann}, throughout this paper that $n \ge 1$.
The \emph{edge number} of $F$ is the quantity $\en(F) = 2k + n - 1$.
If $\tau = (\mathcal{G},\sigma)$ is a reduced marked graph of groups
with exactly one vertex $v$ of valence at least two, 
then the edge number of $F$ is the valence of $v$.

\paragraph{The norm $\|\cdot\|$.}
Given a finite set $W$ of infinite order elements of $F$
and a reduced marked graph of groups $\tau = (\mathcal{G},\sigma)$,
following Krsti\'c and Vogtmann~\cite{KrsticVogtmann}
we define a norm on $\tau$ as
\[  \|\tau\| = \sum_{w \in W} \ell(w), \]
where $\ell(w)$ is the hyperbolic translation length of $w$ on the Bass--Serre tree $T$ of $\mathcal{G}$.
The \emph{star} of a reduced marked graph of groups $\tau$
is the subcomplex of $L$ spanned by those marked graphs of groups $\tau'$
which collapse onto $\tau$.
The \emph{ball of radius $r$}
is the union of the stars of reduced marked graphs of groups $\tau$
with norm at most $r$.
It is contractible (this is the main technical result of~\cite{KrsticVogtmann}).
Krsti\'c--Vogtmann prove~\cite[Proof of Proposition 6.3]{KrsticVogtmann}
that for appropriate choice of $W$,
balls of radius $r$ are compact.
We therefore choose $W$ such that balls of radius $r$ are compact.
This gives us a compact exhaustion of $L$.

\paragraph{Directions and turns.} 
Let $v$ be a vertex of a marked graph of groups $(\mathcal{G},\sigma)$.
We write $\st(v)$ for the set of oriented edges with initial vertex $v$.
The set of \emph{directions} at $v$ is 
\[  D_v = \coprod_{e\in\st(v)} \mathcal{G}_v \times \{e\}.  \]
(Recall that marked graphs of groups have trivial edge groups in this paper.)
There is an obvious action of $\mathcal{G}_v$ on $D_v$;
each orbit is an oriented edge $e \in \st(v)$.
A \emph{turn} is the $\mathcal{G}_v$-orbit of a pair $\{(g_1,e_1),(g_2,e_2)\}$ of directions in $D_v$.
A turn is \emph{degenerate} if it is represented by a pair of identical directions
and \emph{nondegenerate} otherwise.

\paragraph{The star graph.}
Let $\tau = (\mathcal{G},\sigma)$ be a marked graph of groups.
For each $w \in W$ thought of as an element of the fundamental group $\pi_1(\mathbb{G},\star)$,
recall that the homotopy equivalence $\sigma\colon \mathbb{G} \to \mathcal{G}$
induces (after choosing basepoints)
$\sigma_\sharp\colon \pi_!(\mathbb{G}) \to \pi_1(\mathcal{G})$.
Represent the conjugacy class of $\sigma_\sharp(w)$ by a graph-of-groups \emph{edge path}
\[  \gamma_w = g_0e_1g_1\ldots e_kg_k   \]
which is \emph{cyclically reduced} in the sense that 
$e\bar e$ is not a subpath of any cyclic reordering of $\gamma_w$
for any oriented edge $e$ of $\mathcal{G}$.
We say the path $\gamma_w$ \emph{takes} the turns
$[\{(1,\bar e_i),(g_i,e_{i+1})\}]$ for $1 \le i \le k-1$
and $[\{(g_k^{-1},\bar e_k),(g_0,e_1)\}]$.
Since $\gamma_w$ is cyclically reduced, these turns are nondegenerate.
The \emph{star graph of $\tau$} with respect to $W$
has as vertex set the disjoint union of $D_v$ as $v$ varies over the vertices of $\mathcal{G}$
and an edge connecting a pair of directions (necessarily based at the same vertex $v$)
if the turn they determine is taken by some $\gamma_w$.
Thus each turn of $\gamma_w$ based at $v$ corresponds to $|\mathcal{G}_v|$ edges of the star graph.
Since turns taken by $\gamma_w$ are nondegenerate,
the star graph has no loop edges.
If $\tau$ is reduced, then $\|\tau\|$ may be computed from the star graph as
\[  \|\tau\| = \frac{1}{2} \sum_{v\in \mathcal{G}} 
\sum_{d \in D_v} \frac{\operatorname{valence}(d)}{|\mathcal{G}_v|}.   \]

\paragraph{Ideal edges.}
An \emph{ideal edge} based at $v$ in $\mathcal{G}$
is a subset $\alpha$ of $D_v$ with the following properties.
\begin{enumerate}
    \item The sets $\alpha$ and $D_v - \alpha$ have at least two elements.
    \item The set $\alpha$ contains at most one element of each $\mathcal{G}_v$-orbit of directions.
    \item There is a direction $(g,e) \in \alpha$ with the property
        that no direction with underlying edge $\bar e$ belongs to $\alpha$.
        (We may have that $\bar e \notin \st(v)$.)
\end{enumerate}
Write $D(\alpha)$ for the set of directions satisfying item 3
of the above definition.
We are mainly interested in the case where $\tau$ is reduced
and $F = A_1*\cdots*A_n*F_k$ where $n \ge 1$.
In this situation, every vertex of $\mathcal{G}$ has nontrivial vertex group,
so item 2 implies that $D_v - \alpha$ has at least two elements if $\alpha$ has at least two elements.
There is an ideal edge based at $v$ if and only if $v$ has valence at least two,
in which case we say that $v$ is \emph{active.}
The action of $\mathcal{G}_v$ on $D_v$ descends to an action on the set of ideal edges based at $v$,
and we say $\alpha$ and $\alpha'$ are \emph{equivalent} if they belong to the same $\mathcal{G}_v$-orbit.
An ideal edge $\alpha$ is \emph{contained in} an ideal edge $\beta$
if both are based at $v$
and $\alpha$ is equivalent to a subset of $\beta$.
We say $\alpha$ and $\beta$ are \emph{disjoint} 
if they are based at different vertices or if they are based at a vertex $v$
and their $\mathcal{G}_v$-orbits are disjoint.
We say $\alpha$ and $\beta$ are \emph{compatible} if one is contained in the other or they are disjoint.

In the case where $\mathcal{G}_v$ is trivial,
the set $D_v - \alpha$ is also an ideal edge,
and we say $\alpha$ and $D_v - \alpha$ are equivalent,
but compatibility remains the same.

\paragraph{Blowing up ideal edges.}
Given an ideal edge $\alpha$ in a marked graph of groups $\tau$,
we describe a new marked graph of groups $\tau^\alpha = (\mathcal{G}^\alpha,\sigma^\alpha)$
obtained from $\tau = (\mathcal{G},\sigma)$ by \emph{blowing up} the ideal edge $\alpha$.
The graph $\mathcal{G}^\alpha$ is the same as the graph $\mathcal{G}$
with an additional vertex $v_\alpha$ and an additional edge $\alpha$.
The oriented edge $\alpha$ has initial vertex $v$ and terminal vertex $v_\alpha$.
An oriented edge $e$ incident to $v$ in $\mathcal{G}$ is now incident to $v_\alpha$
if there is a direction in the $\mathcal{G}_v$-orbit $e$ in $\alpha$;
all other oriented edges remain incident to the vertices they were incident to in $\mathcal{G}$.
All edge groups remain trivial,
all vertex groups remain what they were in $\mathcal{G}$, and the new vertex has trivial vertex group.
There is a collapse map $\mathcal{G}^\alpha \to \mathcal{G}$
which collapses the edge $\alpha$ and sends an edge $e$ incident to $v_\alpha$
to the edge path $ge$, where $(g,e) \in \alpha$.
This collapse map is a homotopy equivalence;
choose a homotopy inverse $f\colon \mathcal{G} \to \mathcal{G}^\alpha$
and define $\sigma^\alpha = f\sigma$.

If $\tau$ is reduced, 
the marked graph of groups $\tau^\alpha$ represents a vertex of $L$:
collapsing any edge $e$ in $D(\alpha)$ yields a reduced marked graph of groups $\tau^{\alpha}_e$
in which the edge $\alpha$ is not collapsed.
When $\tau$ is reduced,
we say that $\tau^\alpha_e$ is obtained from $\tau$ by the \emph{elementary Whitehead move} $(\alpha,e)$.
When it is convenient, we will sometimes write $(\alpha,d)$ where $d \in D(\alpha)$ is a direction,
rather than an edge.
The vertices $\tau$ and $\tau^\alpha_e$ of $L$ 
are connected by an edge path of length two in $L$
called an \emph{elementary Whitehead path.}
Following~\cite{Vogtmann},
we denote paths in $L$ by listing their vertices;
so the elementary Whitehead path 
from $\tau$ to $\tau^\alpha_e$ is $(\tau,\tau^\alpha,\tau^\alpha_e)$.

If $\alpha$ and $\alpha'$ are equivalent, say $\alpha' = g\alpha$,
then $\tau^\alpha$ and $\tau^{\alpha'}$ are equivalent:
one isomorphism $h\colon \mathcal{G}^\alpha \to \mathcal{G}^{\alpha'}$
is the identity on the common edges and vertices (and vertex groups) 
of $\mathcal{G}^\alpha$ and $\mathcal{G}^{\alpha'}$,
sends $v_\alpha$ to $v_\alpha'$ and $\alpha$ to the edge path $g\alpha'$.

If $\alpha$ is contained in the ideal edge $\beta$, 
then $\alpha$ corresponds to an ideal edge $\alpha$ of $\tau^\beta$
based at $v_\beta$.
If $\alpha$ is based at $v$ and $\beta$ is based at $w$ and they are disjoint,
then $\alpha$ corresponds to an ideal edge $\alpha$ of $\tau^\beta$ based at $v$
and vice versa.
In this latter situation we have ${(\tau^{\alpha})}^\beta = {(\tau^\beta)}^\alpha$.

An \emph{ideal forest} $\Phi = \{\alpha_1,\ldots,\alpha_I\}$ 
in a reduced marked graph of groups $\tau$
is a set of pairwise compatible ideal edges containing at most one element of each equivalence class.
By repeatedly blowing up ideal edges in $\Phi$ 
that are maximal with respect to inclusion,
we obtain a marked graph of groups
$\tau^{\alpha_1,\ldots,\alpha_I}$.
The edges $\{\alpha_1,\ldots,\alpha_I\}$ are a \emph{forest} 
in the sense of~\cite{MyTrainTracks},
and collapsing them recovers $\tau$.

\paragraph{The absolute value of an ideal edge.}
Following Culler and Vogtman~\cite{CullerVogtmann},
define the \emph{dot product} of two subsets $S$ and $T$ 
of the set of vertices of a graph
to be the number of (unoriented) edges with one vertex in $S$ and the other in $T$.
The \emph{absolute value} of $S$ is the dot product of $S$ with its complement.
An ideal edge corresponds to a subset of the set of vertices of the star graph,
so we may compute $|\alpha|$.
If $\alpha' = g.\alpha$, then we have $|\alpha| = |\alpha'|$.
A direction $d$ likewise has an absolute value $|d| = |\{d\}|$.
If $d' = g.d$ for some $g \in \mathcal{G}_v$, then $|d'| = |d|$.
Put another way, the absolute value of a direction 
is actually a property of the underlying oriented edge $e$,
and we will sometimes write $|e|$ for $|d|$.

It follows from~\cite[Lemma 4.7 and Propositions 6.4, 6.5]{KrsticVogtmann}
that if $\Phi = \{\alpha_1,\ldots,\alpha_I\}$ is an ideal forest in a reduced
marked graph of groups $\tau$
and $\{e_1,\ldots,e_I\}$ is a collection of edges in $\tau$
which form a collapsible forest in $\tau^{\alpha_1,\ldots,\alpha_I}$,
after possibly reordering and reversing orientation of the $e_i$,
we have directions $d_i \in D(\alpha_i)$ 
with underlying oriented edges $e_i$ such that
collapsing the edges $e_i$ 
yields a marked graph of groups $\tau^{\alpha_1,\ldots,\alpha_I}_{e_1,\ldots,e_I}$
such that
\[  \|\tau^{\alpha_1,\ldots,\alpha_I}_{e_1,\ldots,e_I}\| 
= \|\tau\| + \sum_{i=1}^I |\alpha_i| - \sum_{i=1}^I |e_i|.
\tag{$\star$}\label{eqnstar} \]

A Whitehead move $(\alpha,e)$ is \emph{reductive}
if $\|\tau^\alpha_e\| \le \|\tau\|$
and \emph{strictly reductive} if the inequality is strict.
By \Cref{eqnstar},
reductivity is equivalent to the condition that $|\alpha| \le |e|$.
An ideal edge $\alpha$ is reductive 
if every Whitehead move supported by that edge is reductive.
The proof of contractibility in~\cite{CullerVogtmann,KrsticVogtmann}
goes by finding reductive ideal edges 
compatible with a given collection of ideal edges.
In our case, we want to \emph{avoid} reductive ideal edges if possible.

If we have $F = A_1*\cdots*A_n*F_k$,
let $s_1,\ldots,s_k$ be a free basis for $F_k$,
and suppose $W$ contains the elements
\[  \{a_i a_j : a_i \in A_i - \{1\},\ a_j \in A_j - \{1\},\ i < j\} \]
assuming $n \ge 2$ along with the elements
\[
    \{s_i a_j,\ a_j s_i : a_j \in A_j\} \quad\text{and}\quad
    \{s_i s_j,\ s_i s_j^{-1} : i \ne j \}.
\]
Note that $W$ therefore contains the $s_i$.
Throughout this paper, we assume that $\out(F)$ is infinite,
i.e.~that $F \ne F_1$, $A_1$, $A_1 * F_1$ or $A_1 * A_2$.

\begin{lem}[cf.~Lemma 2.2 of~\cite{Vogtmann}]\label{absolutenonzero}
    For $W$ as above, the absolute value of each edge and each ideal edge
    in a reduced marked graph of groups $\tau$ is nonzero,
    and the star graph of $\tau$ with respect to $W$
    has exactly one connected component for each vertex of $\tau$.
    Moreover, if $n \ge 2$,
    for each vertex $v$, there exists an oriented edge $e$ with initial vertex $v$
    such that there exists an edge in the star graph between every pair of directions in $D_v$
    with the same underlying edge $e$.
\end{lem}

\begin{proof}
    Let $\tau = (\mathcal{G},\sigma)$ be our reduced marked graph of groups.
    The main idea of the proof is~\cite[6.5 Corollary 2]{Trees},
    which states that if each $w \in W$ is elliptic 
    in (the action of $F$ on the Bass--Serre tree of)
    some graph of groups with fundamental group identified with $F$,
    then actually all of $F$ is elliptic.
    We produce such splittings by possibly blowing up $\tau$
    (although strictly speaking we do not always blow up \emph{ideal edges})
    and then collapsing all but one edge of the resulting graph of groups.
    As long as the blown-up graph of groups is minimal,
    we have a contradiction.

    Let us apply this principle.
    If $|e| = 0$, then each $\sigma_\sharp(w)$ for $w \in W$ does not cross the edge $e$.
    Collapse the other edges of $\mathcal{G}$ to obtain a \emph{free splitting} of $F$,
    that is, a graph of groups with trivial edge groups and fundamental group identified with $F$,
    in which $F$ is elliptic. Since $\mathcal{G}$ was minimal, this is a contradiction,
    so we conclude that $|e| \ne 0$.
    If $\alpha$ is an ideal edge of $\tau$ with $|\alpha| = 0$,
    then thinking of $\alpha$ as an edge of $\tau^\alpha$,
    we can make the same argument to reach a contradiction.
    Therefore $|\alpha| \ne 0$.

    Next we prove that the star graph of $\tau$ with respect to $W$ 
    has one component for each vertex $v$ of $\tau$.
    Write $\Gamma_v$
    for the induced subgraph of the star graph
    comprising all those edges between directions in $D_v$.
    If $v$ has valence one in $\mathcal{G}$, then $n \ge 2$,
    and letting $\mathcal{G}_v \cong A_i$,
    the existence of the words $a_i a_j$ (or their inverses) in $W$
    implies $\Gamma_v$ is connected.

    So suppose $v$ has valence at least two.
    If every component of $\Gamma_v$ has one vertex,
    then in particular some edge of $\mathcal{G}$ satisfies $|e| = 0$,
    which we already saw was impossible.
    So we may suppose some component $C$ of $\Gamma_v$ has at least two vertices.
    Since $\out(F)$ is infinite, assuming $\Gamma_v$ is disconnected,
    $\Gamma_v \setminus C$ has at least two vertices.
    To see this, observe that if $\Gamma_v \setminus C$ is a single direction $d$,
    then $|d| = 0$,
    which implies that the $\mathcal{G}_v$-orbit of $d$ cannot be contained in $C$,
    or, if $\mathcal{G}_v$ is trivial, that $|\bar d| = 0$,
    so $\bar d$ cannot be contained in $C$ if $\mathcal{G}_v$ is trivial.
    In other words, if $\Gamma_v \setminus C$ contains exactly one vertex,
    then it actually contains at least two.
    Write $\mathcal{G}_C$ for the stabilizer of $C$
    in the action of $\mathcal{G}_v$ on $\Gamma_v$.
    Although $C$ does not form an ideal edge in general,
    we may still apply the procedure in the paragraph ``Blowing up ideal edges''
    to produce a new graph of groups $\tau^C = (\mathcal{G}^C,\sigma^C)$
    with one more edge, which we denote as $C$.
    The edge $C$ has edge group $\mathcal{G}_C$.
    The only way $\mathcal{G}^C$ could fail to be minimal is if $\mathcal{G}_C = \mathcal{G}_v$
    and either $v$ or the new vertex $v_C$ has valence one in $\mathcal{G}^C$.
    But together this would imply that either $C$ is empty or all of $\Gamma_v$.
    Collapsing all the edges of $\mathcal{G}^C$ other than $C$
    produces a splitting (not necessarily free) of $F$ in which each $w \in W$ is elliptic,
    a contradiction.
    Therefore we conclude that $\Gamma_v$ is connected in this case.

    For the final statement, observe that topologically the edge path representing $a_i a_j$
    is like a rubber band held at two points:
    the map factors through a two-to-one branched cover of the circle over the interval.
    The branch points correspond to turns between a pair of directions with the same underlying edge.
    As $a_i$ and $a_j$ vary, every vertex of $\tau$ 
    (which has nontrivial vertex group because $\tau$ is reduced)
    appears as the image of one of these branch points.
    In fact, fixing $a_j$ but allowing $a_i$ to vary in $A_i$,
    if the branch point corresponding to $a_i$ yields a turn between a pair of directions
    with underlying oriented edge $e$,
    we see that \emph{every} pair of directions 
    with the same underlying oriented edge equal to $e$ are connected by an edge
    in the star graph.
\end{proof}

For the remainder of the paper, we fix $W$ such that balls are compact
and such that the hypotheses and conclusions of the previous lemma hold.

\section{Connectivity at Infinity}\label{connectivitysection}
We fix our free product
$F = A_1 * \cdots * A_n * F_k$.
Recall from the introduction
our dimension formula
\[
  \dim(L(F)) =
  \begin{cases}
    2k + n - 2 & n \ge 2 \\
    \max\{2k + n - 3, 0\} & \text{otherwise.}
  \end{cases}
\]
Thus it has dimension zero (i.e.\ is a single point)
for $(n,k) = (0,1), (1,0), (1,1)$, and $(2,0)$.
It follows that $\out(F)$ has zero ends (i.e.\ is finite) in these cases.

The complex $L$ has dimension one, i.e.\ is a tree, for $(n,k) = (3,0)$ and $(0,2)$.
Since it is easy to see that this tree is ``bushy'',
i.e.\ is not quasi-isometric to the line,
it follows that these groups, which have more than one end,
actually have infinitely many.
In~\cite{MyselfCAT0},
we showed that additionally the group $\out(A_1*A_2*\mathbb{Z})$
has infinitely many ends,
even though for this choice of $(n,k)$, the complex $L$ has dimension two.

The goal of this section is to show that in all the remaining cases,
the complex $L(F)$ is connected at infinity,
hence the groups $\out(F)$ are one ended.

Let $B_k$ denote the ball of radius $k$.
Since we have chosen $W$ such that balls of radius $k$ are compact,
it is the union of finitely many stars
of reduced marked graphs of groups,
where the \emph{star} of a reduced marked graph of groups $\tau$
is the subcomplex spanned by the set of all $\tau'$ collapsing onto $\tau$.
Since $L$ is locally finite
and the star of $\tau$ is finite,
there are only a finite number of reduced marked graphs of groups
whose stars have nonempty intersection with $B_k$.
Let $C_k$ denote the union of all of these
reduced marked graphs of groups,
and let
\[
  N(k) = \max\{\|\tau\| : \tau \in C_k\}.
\]
From here on out, we will say \emph{briar patch}
for \emph{reduced marked graph of groups.}
To show connectivity at infinity,
we will show that any two briar patches
with norm at least $N(k)$
can be connected by a path that lies outside of $B_k$.

To do this, we start with a path between our two briar patches
of a given form,
and then push it outside of $B_k$ by a homotopy.
We say that a path $P = (\gamma_0,\gamma_1,\ldots,\gamma_\ell)$
is \emph{standard}
if it has even length $\ell = 2m$,
if $\gamma_{2i}$ is a briar patch for $0 \le i \le m$,
and if $(\gamma_{2i},\gamma_{2i+1},\gamma_{2i+2})$
is an elementary Whitehead path for each $i$ satisfying
$0 \le i \le m-1$.
Since $\out(F)$ is generated by Whitehead automorphisms
in the sense of~\cite{CollinsZieschang}
together with the finite group of \emph{factor automorphisms}
that fix a certain briar patch,
given any two briar patches $\tau$ and $\tau'$,
there is a standard path from $\tau$ to $\tau'$.

We will push $P$ toward infinity
by eliminating local minima.
Consider a briar patch $\tau = \gamma_{2i}$
of minimal norm along $P$,
say $\|\tau\| = m$.
The briar patches $\gamma_{2i+2}$ and $\gamma_{2i-2}$
are obtained from $\tau$ by elementary Whitehead moves
$(\alpha,e)$ and $(\beta,f)$ respectively.
By choosing $\tau$ carefully,
we may assume that $\|\tau^\alpha_e\| > m$,
i.e.\ that $|\alpha| > |e|$
and that $\|\tau^\beta_f\| \ge m$,
so that $|\beta| \ge |f|$.
We replace the subpath
$(\tau^\beta_f, \tau^\beta, \tau, \tau^\alpha, \tau^\alpha_e)$
of $P$
by a new standard path from $\tau^\beta_f$ to $\tau^\alpha_e$
that passes through only briar patches of norm strictly greater than $m$.

\begin{lem}\label{mainconnectivitylemma}
  Suppose $\dim(L(F)) \ge 2$ and $F$ is not of the form $A_1*A_2*\mathbb{Z}$.
  Suppose $\tau$ is a briar patch of norm $m$
  and $\tau^\alpha_e$ and  $\tau^\beta_f$ are briar patches
  obtained from $\tau$ by elementary Whitehead moves
  satisfying $\|\tau^\alpha_e\| > \|\tau\|$
  and $\|\tau^\beta_f\| \ge \|\tau\|$.
  There is a standard path from $\tau^\alpha_e$ to $\tau^\beta_f$
  which passes through only briar patches of norm strictly greater than $m$.
\end{lem}

The bulk of this section is given to proving \Cref{mainconnectivitylemma}.
For now we use it to prove the following piece of \Cref{maintheorem}.

\begin{thm}\label{mainconnectivity}
  If $F = A_1 * \cdots * A_n * F_k$
  is such that $\dim(L(F)) \ge 2$
  and not of the form $A_1*A_2*\mathbb{Z}$,
  then $L(F)$ is connected at infinity.
\end{thm}

\begin{proof}
  In the situation of the statement,
  suppose $\rho$ and $\rho'$ are briar patches 
  of norm at least $N(k)$,
  and let $P$ be a standard path from $\rho$ to $\rho'$.
  If $P$ does not already lie outside $B_k$,
  then there is a briar patch $\tau$ of norm $m \le k$ along $P$.
  By assumption on $\rho$ and $\rho'$,
  we may choose $\tau$ such that the adjacent briar patches along $P$,
  call them $\tau^\alpha_e$ and $\tau^\beta_f$,
  which we have already observed are obtained from $\tau$
  by elementary Whitehead moves,
  satisfy the assumptions of \Cref{mainconnectivitylemma}.
  In fact, we may choose $\tau$ to have minimum norm along $P$
  and satisfying these assumptions.
  We may therefore replace the subpath
  $(\tau^\beta_f, \tau^b, \tau, \tau^\alpha, \tau^\alpha_e)$
  by a new standard path from $\tau^\beta_f$ to $\tau^\alpha_e$
  with only briar patches of norm strictly greater than $\tau$ in its interior.

  We may repeat this process.
  At each stage, our new standard path has fewer briar patches
  of minimum norm $m$,
  and since $N(k) > k$,
  we may even increase this minimum until it is bigger than $k$,
  at which point after finitely many repetitions of this process,
  we have a standard path from $\rho$ to $\rho'$
  that avoids $B_k$, as desired.
\end{proof}

The process in \Cref{mainconnectivitylemma}
is the opposite of Collins and Zieschang's ``Peak Reduction''.
Like in~\cite{CollinsZieschang},
the easy case is when the ideal edges $\alpha$ and $\beta$ are compatible.

\begin{lem}\label{compatiblepaths}
  Suppose $\alpha$ and $\beta$
  are compatible ideal edges of the briar patch $\tau$,
  and suppose $(\alpha,e)$ and $(\beta,f)$
  are Whitehead moves such that $(\alpha,e)$ strictly increases
  and $(\beta,f)$ does not decrease norm.
  There is a standard path from $\tau^\alpha_e$ to $\tau^\beta_f$
  with at most one briar patch in its interior
  whose norm is strictly greater than the norm of $\tau$.
\end{lem}

\begin{proof}
  We follow~\cite[Lemma 3.1]{Vogtmann}.
  Consider $(\mathcal{G},\sigma) = \tau^{\alpha,\beta}$.
  We think of $\alpha$ and $\beta$ as edges of $\mathcal{G}$.
  Our standard path lies in the link of $(\mathcal{G},\sigma)$ in $L$
  and we describe it by listing the edges of $\mathcal{G}$
  which must be collapsed.
  So, for example, we have $\tau^\alpha_e = \{\beta,e\}$.

  If $\{e,f\}$ is a (collapsible) forest in $\mathcal{G}$,
  we may take our path to be
  \[
    (\{\beta,e\}, \{e\}, \{e,f\}, \{f\}, \{\alpha,f\}).
  \]
  If $f \ne e$, the middle vertex of this path is $\tau^{\alpha,\beta}_{e,f}$.
  Since $|\alpha| > |e|$ and $|\beta| \ge |f|$,
  \Cref{eqnstar} shows that this briar patch has norm greater than $\|\tau\|$.
  If instead $e = f$,
  this path degenerates to the elementary Whitehead path
  $(\tau^\alpha_e, \tau^{\alpha,\beta}_e, \tau^\beta_e)$.

  Suppose now that $\{e,f\}$ is not a forest in $\mathcal{G}$.
  One case where this happens is when $\alpha = \beta$.
  In this case both $e$ and $f$ belong to $D(\alpha)$,
  so we may take the degenerate path
  $(\tau^\alpha_e,\tau^\alpha,\tau^\alpha_f)$.

  Suppose instead that $\alpha \ne \beta$.
  Since $e$ and $f$ are individually collapsible in $\mathcal{G}$,
  they do not form loops and have at least one endpoint
  with trivial vertex group.
  If $\{e,f\}$ is not a forest,
  the subgraph of groups spanned by $\{e,f\}$
  must be connected and have fundamental group of the form $A_1 * A_2$,
  $\mathbb{Z}$ or $A_1 * \mathbb{Z}$.
  In other words, the subgraph is either a loop on two edges or a path on two edges,
  and at least one vertex common to both edges has trivial vertex group.
  In fact, because $\alpha$ and $\beta$ are distinct,
  the subgraph must contain both $v_\alpha$ and $v_\beta$, 
  both of which have trivial vertex group.
  It follows that the subgraph is the loop on two edges 
  with
  fundamental group $\mathbb{Z}$.
  We see that $\bar e \in D(\beta)$ and $\bar f \in D(\alpha)$.
  In other words, $\{\alpha, e\}$ and $\{\beta, f\}$ are forests.
  If $|e| \ge |f|$, then we may consider the path
  \[
    (\{\beta,e\}\{\beta\}\{\beta,f\},\{f\},\{\alpha,f\}).
  \]
  The briar patch $\tau^\alpha_f$ corresponding to $\{\beta,f\}$
  has norm
  \[ \|\tau^\alpha_f\| \ge \|\tau^\alpha_e\| > \|\tau\|. \]
  if $|f| > |e|$, then we may consider instead the path
  $(\{\beta,e\}, \{e\}, \{\alpha,e\}, \{\alpha\}, \{\alpha,f\})$.
  The briar patch $\tau^\beta_e$ corresponding to $\{\alpha,e\}$
  has norm
  \[
    \|\tau^\beta_e\| > \|\tau^\beta_f\| \ge \|\tau\|. \qedhere
  \]
\end{proof}

As in~\cite{Vogtmann},
the proof of \Cref{compatiblepaths}
gives a simplicial map of a subdivided square
into $L$ which maps the center to $\tau^{\alpha,\beta}$
and the corners to $\tau$, $\tau^\alpha_e$, $\tau^\beta_f$ and a briar patch
$\tau^{\alpha,\beta}_T$ for some maximal forest $T$ in $\tau^{\alpha,\beta}$.
In the case where $\alpha\ne \beta$ and $e \ne f$,
then $\tau^{\alpha,\beta}_T$ is a briar patch with norm strictly greate than $\tau$.
The square, see \Cref{homotopyfig},
provides a homotopy rel endpoints between the old path and the new one.

\begin{figure}[ht]
    \begin{center}
	    \def\svgwidth{\columnwidth}
\begingroup%
  \makeatletter%
  \providecommand\color[2][]{%
    \errmessage{(Inkscape) Color is used for the text in Inkscape, but the package 'color.sty' is not loaded}%
    \renewcommand\color[2][]{}%
  }%
  \providecommand\transparent[1]{%
    \errmessage{(Inkscape) Transparency is used (non-zero) for the text in Inkscape, but the package 'transparent.sty' is not loaded}%
    \renewcommand\transparent[1]{}%
  }%
  \providecommand\rotatebox[2]{#2}%
  \newcommand*\fsize{\dimexpr\f@size pt\relax}%
  \newcommand*\lineheight[1]{\fontsize{\fsize}{#1\fsize}\selectfont}%
  \ifx\svgwidth\undefined%
    \setlength{\unitlength}{544.2519685bp}%
    \ifx\svgscale\undefined%
      \relax%
    \else%
      \setlength{\unitlength}{\unitlength * \real{\svgscale}}%
    \fi%
  \else%
    \setlength{\unitlength}{\svgwidth}%
  \fi%
  \global\let\svgwidth\undefined%
  \global\let\svgscale\undefined%
  \makeatother%
  \begin{picture}(1,0.49479167)%
    \lineheight{1}%
    \setlength\tabcolsep{0pt}%
    \put(0,0){\includegraphics[width=\unitlength,page=1]{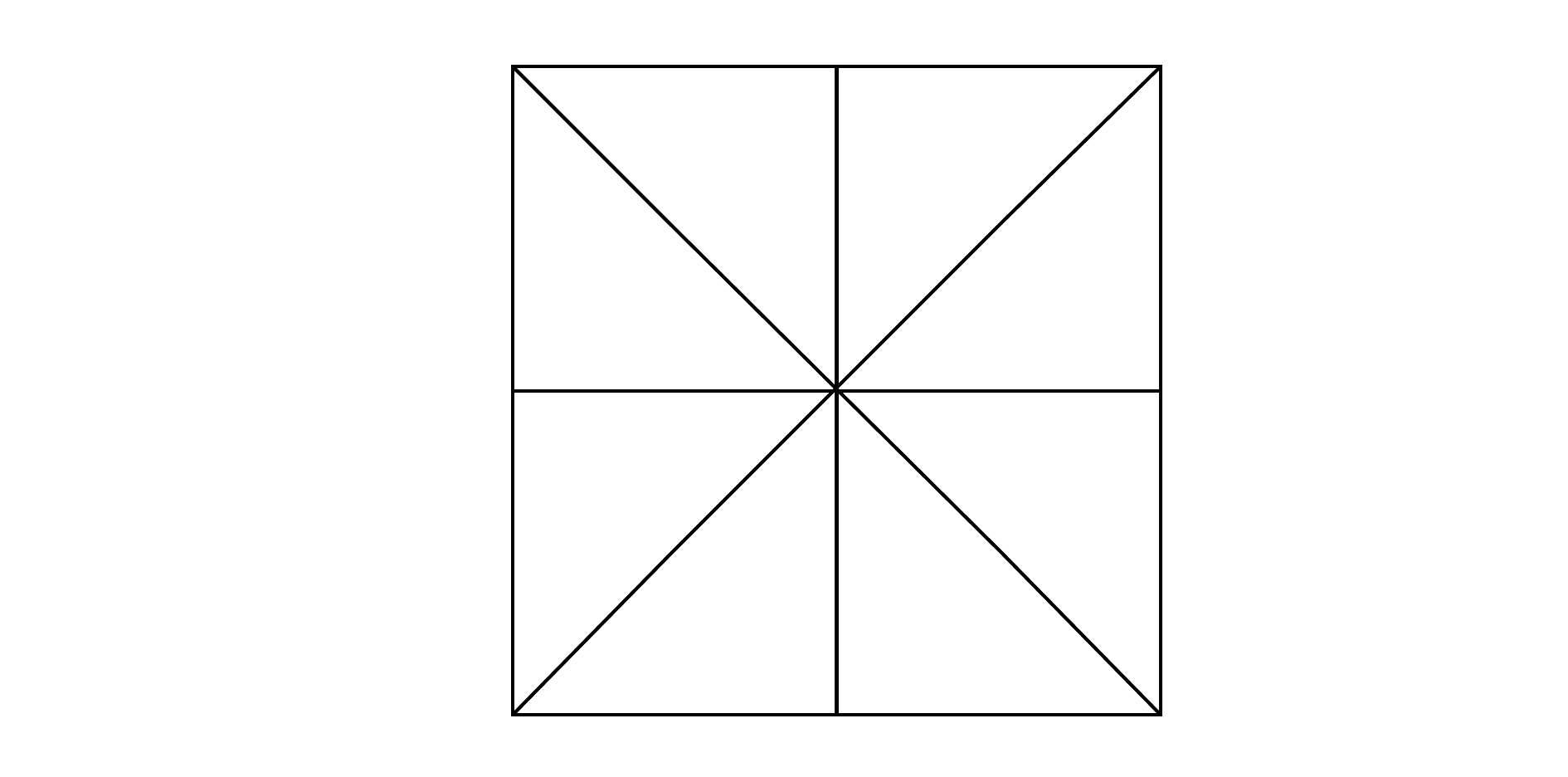}}%
    \put(0.31856793,0.01674302){\color[rgb]{0,0,0}\makebox(0,0)[lt]{\lineheight{1.25}\smash{\begin{tabular}[t]{l}$\tau$\end{tabular}}}}%
    \put(0.72981373,0.01574751){\color[rgb]{0,0,0}\makebox(0,0)[lt]{\lineheight{1.25}\smash{\begin{tabular}[t]{l}$\tau^\alpha_e$\end{tabular}}}}%
    \put(0.31404817,0.47127559){\color[rgb]{0,0,0}\makebox(0,0)[lt]{\lineheight{1.25}\smash{\begin{tabular}[t]{l}$\tau^\beta_f$\end{tabular}}}}%
    \put(0.57880712,0.26926414){\color[rgb]{0,0,0}\makebox(0,0)[lt]{\lineheight{1.25}\smash{\begin{tabular}[t]{l}$\tau^{\alpha,\beta}$\end{tabular}}}}%
    \put(0.73035574,0.46990232){\color[rgb]{0,0,0}\makebox(0,0)[lt]{\lineheight{1.25}\smash{\begin{tabular}[t]{l}$\tau^{\alpha,\beta}_T$\end{tabular}}}}%
  \end{picture}%
\endgroup%

        \caption{The subdivided square 
        in the proof of \Cref{compatiblepaths}.}\label{homotopyfig}
    \end{center}
\end{figure}
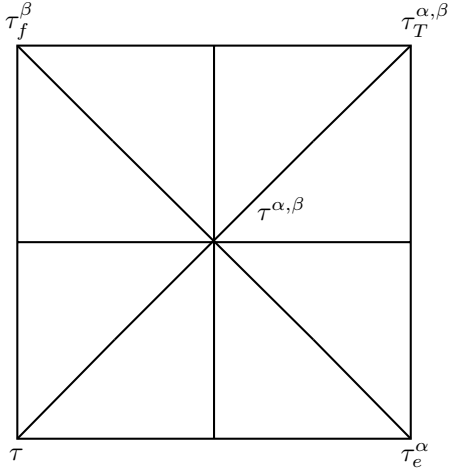

Now, in general, we may not assume that $\alpha$ and $\beta$
are compatible.
The strategy for proving \Cref{mainconnectivitylemma}
is to find a sequence of ideal edges 
$\alpha = \gamma_0,\gamma_1,\ldots,\gamma_n = \beta$
such that each $\gamma_i$ is compatible with its neighbors
and supports a strictly increasing Whitehead move $(\gamma_i,e_i)$.
By \Cref{compatiblepaths},
we obtain a standard path from $\tau^{\gamma_i}_{e_i}$
to $\tau^{\gamma_{i+1}}_{e_{i+1}}$,
and the concatenation of all of these standard paths 
will be the standard path we seek.
We will describe the outline of the argument in a moment.

Firstly, the \emph{size} of an ideal edge is its cardinality,
while the \emph{size} of a Whitehead move is the size of its ideal edge.
(Let us remark that in the case where $n = 0$
and an ideal edge $\alpha$ is equivalent to $D_v - \alpha$,
we should take the size of $\alpha$ to mean the size of the smaller of the two.)
Next, let the \emph{edge number} of $F$, denoted $\en(F)$,
be the quantity $2k + n - 1$.
Assuming that $n \ge 1$, if $\tau$ is a briar patch in $L(F)$ 
with one vertex $\star$
with valence greater than one,
then the valence of $\star$ is $\en(F)$.
Finally, say that a vertex of $\tau$ is \emph{active}
if it has valence at least two.
Active vertices support ideal edges.

In the course of our proof of \Cref{mainconnectivitylemma},
each ideal edge $\gamma_i$ aside from $\alpha$ and $\beta$
will have size two or size three.
Our first step is to reduce to the case where $\tau$ has one active vertex.
\begin{lem}\label{oneactivevertexlemma}
  Suppose $\tau$ has more than one active vertex
  and that the given ideal edges $\alpha$ and $\beta$ are not already compatible.
  Then $\alpha$ and $\beta$ are compatible
  with a size-two ideal edge $\gamma$ (based at a different vertex)
  which strictly increases norm.
\end{lem}

\begin{proof}
  This is a quick corollary of the fact (\Cref{sizetwoexistence} below)
  that when $n \ge 2$,
  which is necessary for $\tau$ to have more than one active vertex,
  each active vertex supports a strictly increasing Whitehead move of size two,
  and that ideal edges based at different vertices are compatible.
\end{proof}

By \Cref{compatiblepaths}, \Cref{mainconnectivitylemma} follows in this case.

Supposing then instead that $\tau$ has only one active vertex,
our first step is the following lemma.
\begin{lem}\label{mainsizetwolemma}
  If $\tau$ is a briar patch supporting the Whitehead move $(\alpha,e)$
  such that $|\alpha| \ge |e|$,
  then either $\alpha$ has size two,
  or $\alpha$ is compatible with a size-two ideal edge
  which supports a strictly increasing Whitehead move.
\end{lem}

\Cref{mainsizetwolemma} allows us to restrict our attention
to the more well-behaved class of size-two Whitehead moves.

Here is our main step.
Consider the simplicially subdivided polygons in \Cref{compatiblefig}.
Further subdivide each one by introducing a new vertex
at the midpoint of each edge in the boundary of the polygon.
Suppose $(\alpha,e)$ and $(\beta,f)$ are Whitehead moves
of size two based at $\tau$
such that $|\alpha| \ge |e|$ and $|\beta| > |f|$.
A \emph{good polygon from $\alpha$ to $\beta$} in the star of $\tau$
is a simplicial map of the further subdivided polygon into $L$
with the following properties.
\begin{enumerate}
  \item The center of the polygon is mapped to $\tau$.
  \item Each vertex of the polygon 
    (three for the triangle, four for the rectangle, six for the hexagon)
    is mapped to a marked graph of groups $\tau^\sigma$
    obtained from $\tau$ by blowing up an ideal edge $\sigma$ of size two.
  \item The marked graphs of groups $\tau^\alpha$ and $\tau^\beta$
    are mapped to vertices of the polygon.
  \item Each midpoint of an edge 
    (three for the triangle, four for the rectangle, six for the hexagon)
    is mapped to a marked graph of groups $\tau^\sigma$
    obtained from $\tau$ by blowing up an ideal edge of size three.
  \item Each size-three ideal edge $\sigma$ 
    is compatible with the size-two ideal edges
    corresponding to the endpoints of the edge which $\sigma$ 
    corresponds to the midpoint of.
  \item At most one size-three ideal edge 
    corresponding to a vertex of the polygon is reductive.
  \item At most one size-two ideal edge (specifically $\alpha$)
    corresponding to a vertex of the polygon is reductive.
\end{enumerate}

\begin{figure}
    \begin{center}
	    \def\svgwidth{\columnwidth}
	        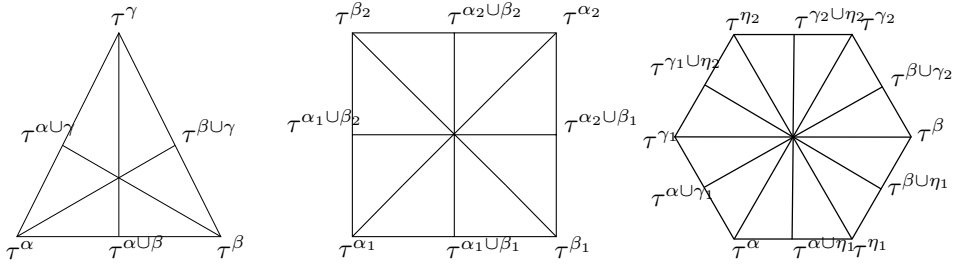
	
        \caption{The simplicially subdivided polygons 
        in the proof of \Cref{mainconnectivity}.}\label{compatiblefig}
    \end{center}
\end{figure}

\begin{lem}\label{maincompatibilitylemma}
  Suppose $\tau$ has one active vertex
  supporting Whitehead moves $(\alpha,e)$ and $(\beta,f)$
  of size two such that $|\alpha| \ge |e|$ and $|\beta| > |f|$.
  Assume that $n \ge 1$, $\en(F) \ge 3$,
  and that $F \ne A_1 * A_2 * \mathbb{Z}$,
  and that $\alpha$ and $\beta$ are not compatible.
  Either (after rechoosing $\alpha$ and $\beta$ in their equivalence class)
  we have that $\alpha \cup \beta$ is an ideal edge which is not reductive, or
  there is a good polygon in the star of $\tau$
  from $\alpha$ to $\beta$.
\end{lem}

With this lemma in hand,
we complete the proof of \Cref{mainconnectivitylemma}.
\begin{proof}[Proof of \Cref{mainconnectivitylemma}]
  We begin with our given Whitehead moves $(\alpha,e)$ and $(\beta,f)$
  as in the statement.
  By \Cref{oneactivevertexlemma} and \Cref{compatiblepaths},
  if $\tau$ has more than one active vertex,
  then there is a standard path from $\tau^\alpha_e$ to $\tau^\beta_f$
  with only briar patches of norm strictly greater than $\|\tau\|$ in its interior.

  If instead $\tau$ has only one active vertex,
  then by \Cref{mainsizetwolemma},
  we may replace $\alpha$ and $\beta$ by ideal edges of size two,
  call them $\alpha'$ and $\beta'$
  that support Whitehead moves
  $(\alpha',e')$ and $(\beta',f')$
  satisfying $|\alpha'| \ge |e'|$ and $|\beta'| > |f|$.
  Then by \Cref{maincompatibilitylemma},
  there is a good polygon in the star of $\tau$.
  Since there is at most one vertex of the polygon 
  (aside from the vertex $\tau^\alpha$)
  corresponding to a reductive ideal edge $\sigma$,
  a path around the polygon from $\tau^{\alpha'}$ to $\tau^{\beta'}$
  that avoids that vertex $\tau^\sigma$
  gives a sequence of ideal edges
  $\alpha' = \gamma_1,\ldots,\gamma_m = \beta'$
  such that each $\gamma_i$ is compatible with its neighbors
  and supports a strictly increasing Whitehead move.
  By \Cref{compatiblepaths}, we are done.
\end{proof}

The remainder of this section is given to proving
the statements left unproved above:
the main lemma left is \Cref{maincompatibilitylemma},
but we also need to prove
\Cref{sizetwoexistence}, and
\Cref{mainsizetwolemma}.

\paragraph{Reduction to one active vertex.}
We turn first to \Cref{sizetwoexistence},
for which our main tool is the following lemma, due to Vogtmann.
We give a proof for completeness.

\begin{lem}[Lemma 3.2 of~\cite{Vogtmann}]\label{threedirections}
  Let $u$, $v$ and $w$ be vertices of a simplicial graph
  (for example the star graph of $\tau$)
  and $S$ the set of other vertices.
  If $|\{u,v\}| \le |v|$ and $|\{u,w\}| \le |w|$,
  then $\{u\} \cdot \{v\} = \{u\}\cdot \{w\}$
  and $\{u\}\cdot S = 0$.
\end{lem}

\begin{proof}
  It is clear that if two sets $S$ and $T$ of vertices of a graph are disjoint,
  then for any set $A$ of vertices, we have
  \[
    A\cdot(S\cup T) = A \cdot S + A \cdot T.
  \]
  Writing $S$ for the set of other vertices as in the statement,
  we have
  \[
    |\{u,v\}| = \{u\} \cdot S + \{u\}\cdot \{w\} + \{v\}\cdot \{w\} + \{v\}\cdot S
  \]
  and
  \[
    |v| = \{v\}\cdot S + \{v\} \cdot \{u\} + \{v\}\cdot \{w\}.
  \]
  Our first inequality implies that
  \[
    \{u\}\cdot S + \{u\}\cdot \{w\} \le \{u\}\cdot \{v\}.
  \]
  Calculating with $w$ in place of $v$,
  we see that
  \[
    \{u\}\cdot S + \{u\}\cdot \{v\} \le \{u\}\cdot \{w\}.
  \]
  Therefore we conclude that $\{u\}\cdot S = 0$
  and $\{u\}\cdot \{v\} = \{u\}\cdot \{w\}$.
\end{proof}

Using \Cref{threedirections}, we prove \Cref{sizetwoexistence}.

\begin{lem}\label{sizetwoexistence}
  Let $\tau$ be a briar patch and suppose that $n \ge 2$.
  Let $v$ be a vertex of $\tau$ of valence at least two.
  There is a strictly increasing Whitehead move of size two based at $v$.
\end{lem}

\begin{proof}
  Since $n \ge 2$, by \Cref{absolutenonzero},
  the star graph of $\tau$ contains an edge betwen a pair of directions
  $d$ and $d'$ in $D_v$ with the same underlying oriented edge $e$.
  Since $v$ has valence at least two,
  there are a pair of ideal edges $\alpha = \{d,c\}$ and $\alpha' = \{d,c'\}$
  where $c$ and $c'$ are distinct and 
  have underlying oriented edges different from $e$ and $\bar e$.
  If both of the Whitehead moves $(\alpha,c)$ and $(\alpha',c')$ are reductive,
  then by \Cref{threedirections},
  we have $d \cdot S = 0$,
  where $S$ denotes the set of remaining vertices in the star graph
  aside from $d$, $c$ and $c'$.
  This contradicts the existence of the edge from $d$ to $d'$.
\end{proof}

\paragraph{Finding increasing Whitehead moves of size two.}
We will now turn to the proof of \Cref{mainsizetwolemma}.
Following Vogtmann, define a \emph{trio} to be a set $\{e,f,g\}$
of three oriented edges with common initial vertex $\star$
of a briar patch $\tau$ which does not contain an edge twice,
once in each orientation.
A trio has the property that $(\{d,d'\},d)$
is an elementary Whitehead move for any choice of directions $d$ and $d'$
representing edges of the trio.
We say the trio \emph{supports} these Whitehead moves.

Vogtmann proves the following lemma using \Cref{threedirections}.

\begin{lem}[Lemma 3.3 of~\cite{Vogtmann}]\label{triosincreasing}
  For every choice of three directions $\{u,v,w\}$ representing
   the oriented edges of the trio $\{e,f,g\}$,
   at least one of the three size-two ideal edges
   $\{u,v\}$, $\{u,w\}$ and $\{v,w\}$
   supports a strictly increasing Whitehead move.
\end{lem}

\begin{proof}
  Suppose all Whitehead moves supported by the ideal edges
  in the statement are reductive.
  Repeatedly applying \Cref{threedirections},
  we see that the directions $u$, $v$ and $w$
  form an isolated component of the star graph.
  This contradicts \Cref{absolutenonzero}.
\end{proof}

%
%

\begin{lem}\label{efbarflemma}
  Suppose $(\alpha,e)$ is a Whitehead move
  satisfying $|\alpha| \ge |e|$.
  If the underlying oriented edges of $\alpha$ 
  are of the form $\{e,f,\bar f\}$,
  and we write $\alpha = \{u,v,\bar v\}$,
  then one of the size-two Whitehead moves
  $\{u,v\}$ or $\{u,\bar v\}$ supports
  a strictly increasing Whitehead move.
\end{lem}

\begin{proof}
  Suppose towards a contradiction
  that the ideal edges $\{u,v\}$ and $\{u,\bar v\}$ are reductive.
  Applying \Cref{threedirections},
  writing $S$ for the set of directions not equal to $u$, $v$ or $\bar v$,
  we see that
  $\{u\}\cdot S = 0$ and $\{u\} \cdot \{v\} = \{u\} \cdot \{\bar v\}$.
  Since 
  \[
    |v|  = \{v\} \cdot \{u\} + \{v\}\cdot \{\bar v\}
    + \{v\} \cdot S = |\bar v|,
  \]
  the latter equality implies that $\{v\}\cdot S = \{\bar v\} \cdot S$.
  We have by assumption $|\{u,v,\bar v\}| \ge |u|$,
  which implies that $\{v\} \cdot S \ge \{u\}\cdot \{v\}$.
  Now,
  \[
    |\{u,v\}| = \{u\}\cdot \{\bar v\} + \{v\}\cdot S
    \quad\text{and}\quad
    |\{u\}| = \{u\}\cdot \{v\} + \{u\}\cdot \{\bar v\},
  \]
  so since we assume $\{u,v\}$ is reductive,
  we conclude that $\{v\}\cdot \{\bar v\} = 0$
  and that $\{v\}\cdot S = \{u\}\cdot \{v\}$.
  In particular $|v| = |u|$.

  Now, either $e$, the underlying oriented edge of $u$,
  forms a loop, or it does not.
  If it does, because $\{u\}\cdot S = 0$,
  by $\mathcal{G}_\star$-equivariance of the star graph,
  every occurence of $e$ (in either orientation)
  in each $w \in W$ is followed or preceded by an occurence of $f$,
  the underlying oriented edge of $v$.
  Since $|e| = |f|$,
  each word has an equal number of occurrences of $e$ and $f$.
  Therefore, if we project to the subspace of
  $H_1(F,\mathbb{Z}/2\mathbb{Z})$ spanned by $e$ and $f$,
  we see that each word $w \in W$
  has image $(0,0)$ or $(1,1)$.
  Since these words do not form a basis for 
  $\mathbb{Z}/2\mathbb{Z} \oplus \mathbb{Z}/2\mathbb{Z}$,
  this contradicts our assumption that the words $w \in W$
  satisfy the hypotheses of \Cref{absolutenonzero}.

  Supposing instead that $e$ does not form a loop,
  then every $w \in W$ contains an even number of occurrences of $e$
  and therefore an even number of occurences of $f$,
  which is again a contradiction mod $2$.
\end{proof}

We are now ready to prove \Cref{mainsizetwolemma}.

\begin{proof}[Proof of \Cref{mainsizetwolemma}]
  Note that if an ideal edge $\alpha$ does not have size two,
  then it contains a trio
  unless its underlying edges are of the form $\{e,f,\bar f\}$.
  If it contains a trio, then $\alpha$
  is compatible with a size-two ideal edge 
  supporting a strictly increasing Whitehead move
  by \Cref{triosincreasing},
  and if not the same holds by \Cref{efbarflemma}.
\end{proof}

We need the following elementary observation.

\begin{lem}\label{whenreductive}
  Suppose $S$ is a nonempty subset of the vertices of a simplicial graph
  containing the vertex $x$.
  We have that $|S| \le |x|$ if and only if we have that
  \[
    \{x\} \cdot T \ge \frac{1}{2} |T|,
  \]
  where $T$ is the set $S - \{x\}$.
  In particular equality holds in one equation
  if and only if it holds in the other.
\end{lem}

\begin{proof}
  Write $T^c$ for the complement of $T$.
  On the one hand, we have $|x| = \{x\} \cdot T + \{x\} \cdot T^c$,
  and on the other, we have
  \[
    |S| = \{x\} \cdot T^c + T \cdot (T^c - \{x\}).
  \]
  A bit of algebra finishes the proof.
\end{proof}

\begin{cor}\label{tworeductive}
  Suppose $T$ is a nonempty subset of the vertices of a simplicial graph.
  There are at most two choices of a vertex $x \notin T$
  such that $|T \cup \{x\}| \le |x|$.
  If there are two choices, say $x$ and $x'$, then
  we have $|T \cup \{x\}| = |x|$ and $|T \cup \{x'\}| = |x'|$
  and in fact elements of $T$ are only connected to each other, $x$ and $x'$.
\end{cor}

\begin{proof}
  By \Cref{whenreductive},
  We have that $|T \cup \{x\}| \le |x|$
  if and only if $\{x\} \cdot T \ge \frac{1}{2} |T|$.
  If this holds for another $x' \notin T$,
  notice that we then have
  \[
    |T| \le \{x\} \cdot T + \{x'\} \cdot T \le |T|.
  \]
  Therefore equality holds throughout and the statement follows.
\end{proof}

For size-two ideal edges, we have the following strengthening.

\begin{cor}\label{onereductivesizetwo}
  Suppose $u$ and $v$ are directions based at a vertex $\star$
  in a briar patch $\tau$
  with valence at least three.
  Suppose that $\star$ has nontrivial vertex group $\mathcal{G}_\star$,
  and that $\{u,v\}$ is an ideal edge.
  There is at most one $h \in \mathcal{G}_\star$
  such that $\{u, h.v\} \sim \{h^{-1}.u,v\}$ is reductive.
\end{cor}

\begin{proof}
  By \Cref{tworeductive}
  applied to $T = \{u\}$,
  we see that there are at most two $h \in \mathcal{G}_\star$
  such that $\{u, h.v\}$ is reductive.
  Supposing that there are two,
  we may without loss of generality suppose that one is $v$
  and the other is $h.v$.
  \Cref{tworeductive} allows us to conclude that $u$ is connected only to
  $v$ and $h.v$ in the star graph.
  But applying \Cref{tworeductive} to $T = \{v\}$,
  we see that $v$ is only connected to $u$ and $h^{-1}.u$.
  It follows that the $\mathcal{G}_\star$-orbits of $u$ and $v$
  together form an isolated component of the star graph.
  This contradicts \Cref{absolutenonzero},
  since $\star$ has valence at least three.
\end{proof}

We also need the following elaboration of \Cref{tworeductive}.

\begin{lem}\label{reallytworeductive}
  Suppose $u,v,w,x,y$ are five distinct vertices of a simplicial graph.
  Suppose that 
  \[
    |\{u,v,w\}| \le \min\{|u|,|v|,|w|\},
    \quad\text{that}\quad
    |\{u,v\}| \ge \min\{|u|,|v|\}
    \quad\text{and}\quad
    |\{u,w\}| \ge \min\{|u|,|w|\}|,
  \]
  and that one of the above inequalities involving $\{u,v\}$ and $\{u,w\}$ is strict.
  If $|\{u,v,x\}| \le |x|$,
  then $|\{u,w,y\}| > |y|$.
\end{lem}

\begin{proof}
  By two applications of \Cref{tworeductive},
  if there exists $y$ such that $|\{u,w,y\}| \le |y|$,
  we have $|w| = |\{u,v,w\}| = |v|$,
  and that $u$ is connected only to $v$ and $w$;
  that is, we have that $|u| = \{u\}\cdot\{v,w\}$.
  But notice that the first chain of equalities
  together with the inequalities in the statement implies that $|v| = |w| \le |u|$.
  So the inequalities in the statement
  imply by \Cref{whenreductive} that $\{u\}\cdot\{v\} \le \frac{1}{2}|u|$
  and that $\{u\}\cdot\{w\} \le \frac{1}{2}|u|$,
  and that one of the inequalities is strict.
  This is a contradiction.
\end{proof}

\paragraph{Finding good polygons: the proof of \Cref{maincompatibilitylemma}.}

For the remainder of this section,
we turn to the proof of \Cref{maincompatibilitylemma}.
Assuming $(\alpha,e)$ and $(\beta,f)$ are size-two Whitehead moves
satisfying $|\alpha| \ge |e|$ and $|\beta| > |f|$
which are not compatible,
there are three cases based on the configurations of $\alpha$ and $\beta$.
We briefly describe the cases.

\begin{enumerate}
  \item \textbf{Case one: $\alpha = \{u,v\}$ and $\beta = \{u,w\}$.}
    In our first case, the directions $\{u,v,w\}$
    represent a trio;
    that is, they are drawn from three distinct \emph{unoriented} edges.

  \item \textbf{Case two: $\alpha = \{u,v\}$ and $\beta = \{u,\bar v\}$.}
    In our second case, we have that $\{u,v,\bar v\}$ is an ideal edge,
    but the three directions $u$, $v$ and $\bar v$
    do not form a trio because $v$ and $\bar v$
    share an underlying \emph{unoriented} edge.

  \item \textbf{Case three: $\alpha = \{u,v\}$ and $\beta = \{u,v'\}$,}
    where $v' = g.v$ for some nontrivial $g \in \mathcal{G}_\star$.
    In our third case, the set of directions $\{u,v,v'\}$ is not an ideal edge.
\end{enumerate}

\paragraph{Case one.}
In this case $\alpha \cup \beta = \{u,v,w\}$ is an ideal edge.
If it is not reductive, we are done.

Supposing $\alpha \cup \beta$ is reductive,
the good polygon we will build is the triangle in \Cref{compatiblefig} left.
The ideal edge $\gamma = \gamma(g)$
will be of the form $\{v,g.w\} \sim \{g^{-1}.v,w\}$
for some nontrivial element $g \in \mathcal{G}_\star$.
Recall that in view of Vogtmann~\cite{Vogtmann},
we assume that $n \ge 1$.
We have the ideal edges $\alpha\cup\gamma(g) = \{u,v,g.w\}$
and $\beta\cup\gamma(g) = \{u,g^{-1}.v,w\}$.
By \Cref{reallytworeductive}, at most one of the size-three ideal edges
$\alpha\cup \gamma(g)$ and $\beta\cup\gamma(g)$ is reductive
as $g$ varies over the nontrivial elements of $\mathcal{G}_\star$.

First, observe that \emph{no} ideal edge of the form $\gamma(g)$ is reductive
when $g \in \mathcal{G}_\star$ is nontrivial.
Indeed, if some $\gamma(g)$ were reductive,
by \Cref{whenreductive}, we have
$\{v\}\cdot\{g.w\} \ge \frac{1}{2}|v|$.
Notice that
\[
  |v| \le \{v\}\cdot\{g.w\} + \{g^{-1}.v\}\cdot\{w\} \le |\{u,v,w\}| \le |v|,
\]
so we conclude that equality holds throughout.
In fact, by replacing $v$ with $w$ (and viewing $\gamma(g)$ as $\{g^{-1}.v,w\}$),
we conclude that $|v| = |w| = |\{u,v,w\}| \le |u|$,
and in particular we have $|u| = \{u\}\cdot\{v,w\}$.
As in the proof of \Cref{reallytworeductive}, this is a contradiction:
since at least one of $\alpha$ and $\beta$ is not reductive,
we have by \Cref{whenreductive} that
\[
  \{u\}\cdot\{v\} + \{u\}\cdot\{w\} < |u|.
\]

Next we show that in fact no ideal edge of the form $\alpha\cup\gamma(g)$
or $\beta\cup\gamma(g)$ is reductive.
Recall that by \Cref{reallytworeductive},
since we assume that $\alpha\cup\beta$ is reductive,
there is at most one such edge,
suppose it is $\alpha\cup\gamma(g)$.
The argument in the case it is of the form $\beta\cup\gamma(g)$ will be identical.
By \Cref{whenreductive}, we have that 
$\frac{1}{2}|\alpha| = \frac{1}{2}|\{u,v\}| 
= \{w\}\cdot\{u,v\} = \{g.w\}\cdot\{u,v\}$.
We will show that $|\alpha\cup\gamma(g)| > |\alpha|$,
contradicting our assumption that this ideal edge is reductive.
Indeed, we claim that
\[
  |\{u,v,g.w\}| > 2\{u,v\}\cdot\{w\} = |\alpha|.
\]
Indeed, observe that $\{g.w\} \cdot \{g.u,g.v\} = \{w\}\cdot\{u,v\}$
by $\mathcal{G}_\star$-equivariance of the star graph,
so it is clear that we have inequality.
That the inequality is \emph{strict}
follows in the case where $\en(F) > 3$
by \Cref{absolutenonzero},
since the $\mathcal{G}_\star$-orbits of the directions $u$, $v$ and $w$
cannot form an isolated component of the star graph.
If instead $\en(F) = 3$, so that $n \ge 2$,
then again by \Cref{absolutenonzero},
we must have $\{w\}\cdot \{g.w\} > 0$.

Thus \emph{any} choice of $g \ne 1$ in $\mathcal{G}_\star$
yields a good polygon from $\alpha$ to $\beta$.

\paragraph{Case two.}
Recall that in this case we assume that $\alpha = \{u,v\}$ 
and $\beta = \{u, \bar v\}$,
i.e.\ that $\alpha\cup\beta$ is an ideal edge,
but $v$ and $\bar v$ share an underlying \emph{unoriented} edge.
If $\alpha \cup\beta$ is not reductive, we are done.
Supposing this ideal edge is reductive,
the good polygon from $\alpha$ to $\beta$ we will build will be the rectangle
in \Cref{compatiblefig} center,
where we set $\alpha_1 = \alpha$ and $\beta_1 = \beta$.

To construct this good polygon, 
we need size-two ideal edges $\alpha_2$ and $\beta_2$
so that $\alpha_1\cup\beta_2$, $\alpha_2\cup\beta_1$ and $\alpha_2\cup\beta_2$
are size-three ideal edges
(where we may abuse notation slightly so that these unions may be taken)
and so that none of these five ideal edges (including $\alpha_2$ and $\beta_2$)
are reductive.

Since we assume that $F \ne A_1 * A_2 * \mathbb{Z}$,
there exists a direction $w$ whose $\mathcal{G}_\star$-orbit is disjoint
from $\{u,v,\bar v\}$;
we choose one such.
For $h \in \mathcal{G}_\star$,
we focus on the ideal edges
$\alpha(h) = \{u, h.v\}$,
$\beta(h) = \{u, h.\bar v\}$,
$\gamma(h) = \{\bar v, h.w\}$
and $\delta(h) = \{v, h.w\}$.
We have $\alpha = \alpha(1)$ and $\beta = \beta(1)$
and will choose $\alpha_2$ from $\{\alpha(g), \gamma(1), \gamma(g)\}$
and $\beta_2$ from $\{\beta(g), \delta(1), \delta(g)\}$
for $g \in \mathcal{G}_\star$ nontrivial.
Observe that by abusing notation slightly,
for any choice of $\alpha_2$ and $\beta_2$,
the size-three ideal edges $\alpha\cup\beta_2$, $\alpha_2\cup\beta$ and $\alpha_2\cup\beta_2$ 
make sense to define.

By \Cref{onereductivesizetwo}, at most one of $\gamma(1)$ and $\gamma(g)$ is reductive,
and similarly for $\delta(1)$ and $\delta(g)$.
Since we assume that $\{u,v,\bar v\}$ is reductive,
neither $\alpha(g)$ nor $\beta(g)$ can be reductive:
if, say, $|\alpha(g)| \le |v|$,
we have that $\{u\}\cdot\{g.v\} \ge \frac{1}{2}|u|$ by \Cref{whenreductive},
but this would that
\[
  |u| \le \{u\}\cdot\{g.v\} + \{v\}\cdot\{g^{-1}.u\} < |\{u,v,\bar v\}|,
\]
contradicting our assumption that this ideal edge is reductive.

Notice that if $\gamma(1)$ is reductive,
we have in particular that $\{\bar v\}\cdot\{w\} \ge \frac{1}{2}|w|$.
Then all four of the ideal edges $\beta(1)\cup\gamma(g)$, 
$\beta(g)\cup\gamma(g)$, $\delta(1)\cup\gamma(g)$
and $|\delta(g)\cup\gamma(g)$
are greater in absolute value than
$\{\bar v\}\cdot\{w\} + \{g.\bar v\}\cdot\{g.w\} \ge |w|$
and so are not reductive.

From this it follows that if there are two reductive size-two ideal edges
under consideration,
there is a good polygon from $\alpha$ to $\beta$.
For instance, if $\gamma(1)$ and $\delta(g)$ are reductive,
the relevant path around the good polygon reads
\[
  \alpha,\ \alpha\cup\delta(1),\ \delta(1),\ 
  \gamma(g)\cup\delta(1),\ \gamma(g),\ \beta\cup\gamma(g),\ \beta.
\]

We proceed continuing to assume that $\gamma(1)$ is reductive,
but now assuming that neither $\delta(1)$ nor $\delta(g)$ is reductive.
If one of $\alpha\cup\beta(g)$, $\alpha\cup\delta(1)$ or $\alpha\cup\delta(g)$ is not reductive,
we may use that ideal edge together with ideas above 
to construct a good polygon from $\alpha$ to $\beta$.
Explicitly, if $\alpha\cup\delta(1)$ is not reductive,
the relevant path reads
\[
  \alpha,\ \alpha\cup\delta(1),\ \delta(1),\ \delta(1)\cup\gamma(g),\
  \gamma(g),\ \beta\cup\gamma(g),\ \beta.
\]
We will show that having all three ideal edges reductive leads to a contradiction.
In fact, this does not use the assumption that $\gamma(1)$ is reductive.
Indeed, by \Cref{tworeductive},
we conclude assuming that $\alpha\cup\beta$ and $\alpha\cup\beta(g)$ are reductive
that $|v| \le |u|$.
If we had equality, \Cref{tworeductive}
would additionally imply that $\alpha\cup\delta(1)$ and $\alpha\cup\delta(g)$
are greater in norm than $|w|$, in contradiction to our assumption
that these ideal edges are reductive.
Again by \Cref{tworeductive},
we conclude that $|w| \le |v|$ as well.
Therefore in particular by \Cref{whenreductive}, we have $|\alpha| = \{w, g.w\}\cdot\{u,v\} \le |w|$.
Since therefore $|w| \ge |\alpha| \ge |v|$, we conclude that equality holds throughout.
But it cannot, since this would imply that the $\mathcal{G}_\star$-orbits
of the directions $u$, $v$ and $w$ form an isolated component of the star graph,
in contradiction to \Cref{absolutenonzero}.

Finally, assume that none of the ideal edges $\gamma(1)$, $\gamma(g)$,
$\delta(1)$ or $\delta(g)$ are reductive.
Now by \Cref{tworeductive}, at most two of the four ideal edges
$\alpha\cup\beta_2$, which take the form $\{u,v,x\}$,
are less than $|x|$ in norm.
If we have $|u| < |v|$, since $\alpha\cup\beta$ is reductive,
this is the only ideal edge of that form which is reductive,
and similarly for ideal edges of the form $\{u,\bar v,x\}$.
It remains, then, just to show that one of the ideal edges
of the form $\alpha_2 \cup\beta_2$ is not reductive.
Suppose, for instance, then, that $\alpha(g) \cup\beta(g)$ is reductive.
This, by \Cref{tworeductive}, implies that
$|\{v,\bar v\}| = \{u, g^{-1}.u\}\cdot\{v,\bar v\} \le |u|$.
Since $|u| < |v|$, this implies that $|\{v,\bar v\}| < |v|$,
so that $\{v\}\cdot\{\bar v\} > \frac{1}{2}|v|$,
and therefore $\gamma(1)\cup\delta(g)$ and $\gamma(g)\cup\delta(1)$
are both greater in norm than $|v|$,
and thus are not reductive unless $|v| < |w|$.
But notice that because $|u| < |v|$,
reductivity of $\alpha(g) \cup\beta(g)$ implies that
$\{u, g.v, g.\bar v\} < |g.\bar v| = |v|$;
by \Cref{tworeductive} it is the only ideal edge of the form $\{u, g.v, x\}$
to be less than or equal to $|x|$ in norm.
In particular $|\alpha(g) \cup \delta(1)| > |w|$ and so is not reductive.
This provides a good polygon from $\alpha$ to $\beta$ in this case.

We therefore assume that $|v| \le |u|$.
Now if $|w| < |v|$,
for each of the ideal edges $\gamma(1)$, $\gamma(g)$, $\delta(1)$, and $\delta(g)$
at most of the relevant size-three ideal edges containing it can be reductive
by \Cref{tworeductive}.
The only worry, then, is that both $\gamma(1)\cup\beta$ and $\gamma(g)\cup\beta$
are reductive or,
symmetrically, both $\alpha\cup\delta(1)$ and $\alpha\cup\delta(g)$ are reductive.
But in either case this implies that $|v| \le |w|$:
for instance in the latter case we have $|v| \le |\{u,v\}| = \{w, g.w\}\cdot\{u, v\} \le |w|$.
This contradiction provides a good polygon from $\alpha$ to $\beta$ in this case.

Continuing to assume that $|v| \le |u|$,
we now suppose that $|v| \le |w|$.
Since the argument in the previous case
that not all three ideal edges of the form $\alpha\cup\beta_2$
are reductive did not use the assumption that $\gamma(1)$ was reductive,
we conclude that at least one ideal edge of the form $\alpha\cup\beta_2$ is not reductive.
Suppose at first that we may choose $\beta_2 = \delta(1)$.
(By relabelling if necessary, this is equivalent to choosing $\beta_2 = \delta(g)$.)
Now by symmetry we know that at least one ideal edge of the form $\alpha_2\cup\beta$
is not reductive,
but we can show more:
If both $\delta(1)\cup\beta$ and $\delta(g)\cup\beta$ are reductive,
then $|v| = |w| = \{w,g.w\}\cdot\{u,\bar v\} = |\{u, \bar v\}|$,
which is a contradiction to \Cref{absolutenonzero},
(as well as the assumption that $\beta$ is not reductive)
since it implies that the $\mathcal{G}_\star$-orbits of $u$, $\bar v$ and $w$ 
form an isolated component of the star graph.

Therefore either $\delta(1)\cup\beta$ or $\delta(g)\cup\beta$ is not reductive.
Suppose at first the former is not reductive.
If $\delta(1)\cup\gamma(1)$ is not reductive, we are done.
Supposing it is,
we see by \Cref{whenreductive} that $\{u,w\}\cdot\{v,\bar v\} = |\{v,\bar v\}|$.
This implies that
$|\{v,\bar v,g.w\}| = |w| + |\{v,\bar v\}|$,
that $|\{g^{-1}.u,v,g.w\}| \ge |v| + |\{v, \bar v\}|$,
that $|\{g^{-1}.u,\bar v,g.w\}| \ge |v| + |\{v, \bar v\}|$,
and that $|\{g^{-1}.u,v,\bar v\}| = |u| + |\{v, \bar v\}|$
so these ideal edges are not reductive.
If we can show that one of $\alpha\cup\beta(g)$ or $\alpha\cup\delta(g)$ is not reductive
and that one of $\alpha(g)\cup\beta$ and $\gamma(g)\cup\beta$ is not reductive,
we will have a good polygon from $\alpha$ to $\beta$.

Suppose, then, that $\alpha\cup\beta(g)$ and $\alpha\cup\delta(g)$ are reductive.
From the former, we conclude that 
$\{u\}\cdot\{v\} = \{u\}\cdot\{v,g.\bar v\} \ge \frac{1}{2}|\{v,g.\bar v\}| = |v|$,
so equality holds throughout.
This implies that $|\bar v| = \{w\}\cdot\{\bar v\}$.
From reductivity of the latter,
we conclude using \Cref{whenreductive} that $|\{u,g.w\}| = \{u\}\cdot\{v\} + \{w\}\cdot\{\bar v\}$.
In particular $|u| = \{u\}\cdot\{v,g.w\}$ and $|w| = \{w\}\cdot\{g^{-1}.u,\bar v\}$.
This implies that $|\{u,g.v,\bar v\}| = |u| + 2|v|$
and that $|\{g^{-1}.u,v,w\}| = 3|v|$,
so neither of these ideal edges are reductive,
and we have a good polygon from $\alpha$ to $\beta$,
the relevant path around which reads 
\[
  \alpha,\ \alpha\cup\delta(1),\ \delta(1),\ \alpha(g)\cup\delta(1),\ \alpha(g)\cup\beta,\ \beta.
\]

A symmetric argument dispenses with the case in which both
$\alpha(g)\cup\beta$ and $\delta(g)\cup\beta$ are reductive.
We have two nested assumptions that need to be dispensed with:
the outer one, that $\alpha\cup\delta(1)$ is not reductive,
and the inner one, that $\gamma(1)\cup\beta$ is not reductive.

So, still supposing that $\alpha\cup\delta(1)$ is not reductive,
suppose that $\gamma(1)\cup\beta$ is reductive.
As we noted above, this implies that $\gamma(g)\cup\beta$ is not reductive,
and therefore if $\gamma(g)\cup\delta(1)$ is not reductive,
we are done.
We therefore suppose that it is reductive.
We first show that $\gamma(g)\cup\delta(g)$ is not reductive.
Indeed, if it were, we see by \Cref{tworeductive}
that $\{u,g.w\}\cdot\{v,\bar v\} = |\{v,\bar v\}|$.
Since $\gamma(g)\cup\delta(1)$ is reductive,
we conclude that in fact
$\{g.w\}\cdot\{\bar v\} = \{g.w\}\cdot\{g.v,\bar v\} \ge \frac{1}{2}|\{v,g^{-1}.v\}| = |\bar v|$,
so equality holds throughout.
But this contradicts our assumption that $\gamma(1)\cup\beta$ is reductive,
since by \Cref{whenreductive} that implies that
$\{\bar v\}\cdot\{u,w\} \ge \frac{1}{2}|\{u,w\}| > 0$.

Thus we are done if $\alpha\cup\delta(g)$ is not reductive.
So suppose it is.
Applying \Cref{whenreductive}, we see that
\[
  |u| + |w| \le \{u\} \cdot \{v, \bar v\} + \{w\}\cdot\{g^{-1}.v, \bar v\}
  + \{u\}\cdot\{w\} + \{u\}\cdot\{g.w\} \le |u| + |v| - \{u\}\cdot\{w\} - \{g^{-1}.u\}\cdot\{w\},
\]
so we have equality throughout; in particular $\{u\}\cdot\{w\} = \{g^{-1}.u\}\cdot\{w\} = 0$.
But this contradicts reductivity of $\gamma(g)\cup\delta(1)$,
since we must have $\{w\}\cdot\{v,g^{-1}.\bar v\} \ge \frac{1}{2}|\{v,g^{-1}.\bar v\}| > 0$.
Therefore we again have a good polygon from $\alpha$ to $\beta$.

Finally, we supposed that we could choose $\alpha\cup\delta(1)$ to not be reductive.
To negate this, in view of the possibility of relabelling,
we may assume that $\alpha\cup\delta(1)$ and $\alpha\cup\delta(g)$ are both reductive.
This implies that $|w| = |v|$,
and thus that $\{w, g.w\} \cdot\{u,v\} = |\{u,v\}|$.
As we saw, we must have that $\alpha\cup\beta(g)$ is not reductive.
Since this is equal in norm to $|\{u,v\}| + |v|$,
we have that these sum to more than $|u|$.
By \Cref{whenreductive},
this implies that $\{u\}\cdot\{v\} < |v|$.
But reductivity of $\alpha\cup\beta$
actually implies that
$\{u\}\cdot\{v\} = \{u\}\cdot\{v,\bar v\} \ge \frac{1}{2}|\{v,\bar v\}| = |v|$.
This contradiction completes the proof in this case.

\paragraph{Case three.}
Recall that in this case we assume that  $\alpha = \{u,v\}$
and $\beta = \{u,g.v\}$,  i.e.\ $\alpha$ and $\beta$
share both underlying oriented edges.
In this case  $\alpha \cup \beta$ fails to be an ideal edge.
Since $\en(F) \ge 3$, there exists a direction $w$ based at $\star$
whose $\mathcal{G}_\star$-orit is distinct from that of $u$ and $v$;
choose one.

The proof breaks into two cases:
either we are able to choose $w$ so that
$\{u,v,w\}$ represents a trio,
or we are not able to,
in which case we may choose $w$ to either be $\bar u$ or $\bar v$,
since we assume that $F \ne A_1 * A_2 * \mathbb{Z}$.

\paragraph{The first subcase.}
Suppose at first that $\{u,v,w\}$ represents a trio.
We consider the ideal edges $\gamma(1) = \{u,w\}$,
$\gamma(g) = \{u,g.w\}$,
$\delta(1) = \{v,w\}$ and $\delta(g) = \{v,g.w\}$.
We will use these ideal edges to build a good polygon
(the rectangle in \Cref{compatiblefig})
from $\alpha$ to $\beta$.

Observe that we cannot have that $\alpha\cup\gamma(1)$ and $\alpha\cup\gamma(g)$
are both reductive.
If we did, we conclude by \Cref{tworeductive}
that $|\alpha| = \{w, g.w\}\cdot\{u,v\}$.
An argument based on \Cref{absolutenonzero}
shows that both $\alpha\cup\gamma(1)$ and $\alpha\cup\gamma(g)$
must be strictly greater in norm than this quantity,
which in turn is not less than the minimum of $|u|$ and $|v|$,
implying neither ideal edge is reductive after all.
The same argument applies to all the pairs
$\alpha\cup\delta(1)$ and $\alpha\cup\delta(g)$,
$\beta\cup\gamma(1)$ and $\beta\cup\gamma(g)$
and $\beta\cup\delta(1)$ and $\beta\cup\delta(g)$.

In fact, provided that, say, $\gamma(1)$ is not reductive,
the same argument above proves that $\alpha\cup\gamma(1)$
and $\beta\cup\gamma(1)$ cannot both be reductive.
By \Cref{onereductivesizetwo},
at most one of $\gamma(1)$ and $\gamma(g)$ is reductive,
and similarly for $\delta(1)$ and $\delta(g)$.

So suppose for instance that $\gamma(1)$ and $\delta(g)$ are not reductive.
At most one of $\alpha\cup\gamma(1)$ and $\alpha\cup\gamma(g)$ can be reductive;
suppose the latter is not.
If $\beta\cup\gamma(g) = \beta\cup\delta(1)$ is reductive,
then we have that both $\alpha\cup\delta(g) = \alpha\cup\gamma(g)$
and $\beta\cup\delta(g)$ are not reductive.
This actually suffices to prove the theorem,
but we can work a little harder to complete this to a good polygon.
If both $\gamma(g)$ and $\delta(1)$ are reductive,
we have $\{w\}\cdot\{g^{-1}.u,v\} = |w|$.
This implies that $|\beta\cup\gamma(1)| = |\{u,g.v,w\}| > 2|w|$,
so is not reductive.
(In fact, with a little more care, it can be shown that the
ideal edge $\alpha\cup\gamma(1)$ is also not reductive,
so the good polygon with corners $\alpha$, $\gamma(1)$, $\beta$, and $\delta(g)$ has no
reductive size-three ideal edges.)
If, say, $\delta(1)$ is not reductive,
continuing to assume that $\beta\cup\gamma(g) = \beta\cup\delta(1)$ is reductive,
we conclude that $\alpha\cup\gamma(1) = \alpha\cup\delta(1)$ is not,
and we have a good polygon with corners
$\alpha$, $\delta(1)$, $\beta$ and $\delta(g)$.
A variant of this argument works in all cases.

\paragraph{The second subcase.}
So suppose instead that $F = A_1 * F_2$:
our only choices for $w$ are $\bar u$ and $\bar v$.
We will consider the ideal edges
$\gamma(1) = \{u,\bar v\}$,
$\gamma(g) = \{u,g.\bar v\}$,
$\delta(1) = \{\bar u, v\}$,
$\delta(g) = \{\bar u, g.v\}$,
$\eta(1) = \{\bar u, \bar v\}$
and $\eta(g) = \{\bar u, g.\bar v\}$.
The good polygon we build will be either the rectangle or the hexagon
in \Cref{compatiblefig}, right,
with $\alpha$ and $\beta$ seprarated by one vertex of the polygon.

Now, by \Cref{onereductivesizetwo},
at most one $\gamma$ ideal edge,
one $\delta$ ideal edge and one $\eta$ ideal edge can be reductive.
Suppose that there are three such.
For simplicity assume they are $\gamma(1)$,
$\delta(1)$ and $\eta(1)$:
the other cases are essentially identical.
Then by \Cref{whenreductive}, we have that
$\{\bar v\}\cdot\{u, \bar u\} = |\bar v|$,
while $\{\bar u\}\cdot\{v, \bar v\} = |\bar u|$.
In particular $|u| = |v|$.
Observe, then, that
each of the ideal edges
$\alpha\cup\gamma(g)$,
$\beta\cup\gamma(g)$,
$\alpha\cup\delta(g)$
and $\beta\cup\delta(g)$ are greater in norm than
$\frac{3}{2}|v| = \frac{3}{2}|u|$,
so none of these ideal edges are reductive
and we have a good rectangle from $\alpha$ to $\beta$.

Suppose next that there are two reductive
size-two ideal edges.
Up to swapping $u$ and $v$,
we will consider the cases of $\gamma(1)$ and $\eta(1)$
as well as $\gamma(1)$ and $\delta(1)$;
the others are essentially identical.
If $\gamma(1)$ and $\eta(1)$ are reductive,
then by \Cref{whenreductive},
we have that $|\bar v| = \{\bar v\}\cdot\{u,\bar u\}$.
We claim that we have a good rectangle
with corners
\[
  \alpha,\ \gamma(g),\ \beta\text{ and }\delta(1).
\]
To see this, suppose that one of the relevant size-three ideal edges
is reductive. For example, if $\alpha\cup\gamma(g)$ is reductive,
we have
\[
  |v| \ge \{u\}\cdot\{v\} = \{u\}\cdot\{v, g.\bar v\} \ge 
  \frac{1}{2}|\{v,g.\bar v\}| = |v|.
\]
It follows that $|\beta\cup\gamma(g)| = |u| + |v| + |\bar v|$
and $|\beta\cup\delta(1)| = |v| + |\{g^{-1}.u,\bar u\}|$,
so neither of these ideal edges is reductive.
And additionally, by \Cref{whenreductive},
if $\alpha\cup\delta(1)$ is reductive,
then $|\bar u| = \{\bar u\}\cdot \{u,\bar v\} < \{u\}\cdot\{u,v,\bar v\} = |u|$,
which is impossible.
The argument if $\beta\cup\gamma(g)$ is reductive is analogous.
On the other hand, if $\alpha\cup\delta(1)$ is reductive,
it follows from \Cref{whenreductive}
that $|v| = \{v\}\cdot\{u,\bar u\} = \frac{1}{2}|\{u,\bar u\}|$.
Then $|\beta\cup\gamma(g)| = |u| + 2|v|$, so this ideal edge is not reductive.
Reductivity of $\beta\cup\delta(1)$ or $\alpha\cup\gamma(g)$
would imply respectively that 
$\{v\}\cdot\{\bar u\} = |v| = |u|$
or $\{v\}\cdot\{u\} = |v|$.
But again this would imply in either case 
that $|u|$ and $|\bar u|$ cannot be equal, 
but they must.
The case that $\beta\cup\delta(1)$ is reductive is again analogous.
The existence of a good rectangle follows.

If instead $\gamma(1)$ and $\delta(1)$ are reductive,
we again have a good rectangle,
this time with corners
\[
  \alpha,\ \gamma(g),\ \beta\ \text{ and }\delta(g).
\]
Indeed, \Cref{whenreductive}
implies that $\{\bar v\}\cdot\{u\} \ge \frac{1}{2}|u|$ and $\frac{1}{2}|\bar v|$,
while $\{v\}\cdot\{\bar u\} \ge \frac{1}{2}|\bar u|$ and $\frac{1}{2}|v|$.
We observe then, that
$|\alpha\cup\gamma(g)| \ge \frac{3}{2}|u|$
because we have
\[
  \{u\}\cdot\{\bar v\} + \{v\}\cdot\{\bar u\} + \{g.\bar v\}\cdot\{g.u\} 
  \le|\{u,v,g.\bar v\}.
\]
Similarly we see that $|\alpha\cup\delta(g)| \ge \frac{3}{2}|v|$
and $|\beta\cup\delta(g)| \ge \frac{3}{2}|v|$,
so none of these ideal edges are reductive.

If there is only one reductive size-two ideal edge,
up to swapping the roles of $u$ and $v$ it is either
a $\gamma$ or an $\eta$.
Suppose that the reductive ideal edges is $\gamma(1)$,
so that by \Cref{whenreductive}, we have $\{u\}\cdot\{\bar v\} \ge \frac{1}{2}|v|$
and $\frac{1}{2}|u|$.
Since the $\mathcal{G}_\star$-orbits of $u$, $v$ and $\bar v$
cannot form an isolated component of the star graph, we conclude that
$|\beta\cup\gamma(g)| > |u|$ and $|\alpha\cup\gamma(g)| > |u|$.
The same argument applies to show that
$|\alpha\cup\delta(g)| > |v|$
and $|\beta\cup\delta(g)| > |v|$,
so we have a good rectangle from $\alpha$ to $\beta$.
The argument supposing $\gamma(g)$ is reductive is identical.

If instead $\eta(1)$ is reductive,
we see by \Cref{whenreductive} that $\{\bar u\}\cdot\{\bar v\} \ge \frac{1}{2}|v|$
and $\frac{1}{2}|u|$.
This implies that $\alpha\cup\gamma(1)$ and $\beta\cup\gamma(g)$
cannot both be reductive,
for if they were, we would have
$|\{v,\bar v\}| = \{u,g^{-1}.u\}\cdot\{v,\bar v\}$
in contradiction to the fact that $\{\bar v\}\cdot\{\bar u\} > 0$.
Similarly at most one of $\alpha\cup\gamma(g)$ and $\beta\cup\gamma(1)$,
one of $\alpha\cup\delta(1)$ and $\beta\cup\delta(g)$
and one of $\alpha\cup\delta(g)$ and $\beta\cup\delta(1)$ are reductive. 

Now by \Cref{tworeductive},
at most two ideal edges of the form $\alpha\cup\{x\}$
satisfy $|\alpha\cup\{x\}| \le |x|$,
and if there are two, equality holds.
If $|u| < |v|$, we see that $\alpha\cup\gamma(1)$ and $\alpha\cup\gamma(g)$
cannot both be reductive.
If $|u| > |v|$, the same applies to $\alpha\cup\delta(1)$ 
and $\alpha\cup\delta(g)$,
and if $|u| = |v|$,
at most two of these four ideal edges may be reductive.
The same holds for size-three ideal edges containing $\beta$.

Suppose $|u| < |v|$.
If $\alpha\cup\gamma(1)$ and $\beta\cup\gamma(1)$ are reductive,
then $\alpha\cup\gamma(g)$ and $\beta\cup\gamma(g)$ are not,
and if $\alpha\cup\delta(1)$ is reductive,
then $\beta\cup\delta(g)$ is not
and regardless of whether $\alpha\cup\delta(g)$ is,
we have a good rectangle from $\alpha$ to $\beta$.
The same proof works if $\alpha\cup\gamma(g)$ and $\beta\cup\gamma(g)$ 
are reductive.
If neither of these occurs,
then at most one of the four size-three ideal edges
is reductive and we again have a good rectangle.

If we suppose $|v| < |u|$,
swapping the roles of $u$ and $v$ completes the proof as above.
So suppose $|v| = |u|$.
Now at most two of the four size-three ideal edges containing $\alpha$
are reductive, and similarly for $\beta$.
If there are three or fewer, or more generally,
if for some $\gamma$ or some $\delta$,
the relevant size-three ideal edges are both not reductive,
we have a good rectangle.
If there are four, each involving a different $\gamma$ or $\delta$,
we can build a hexagon.
Indeed, because three of the four directions cannot form
an isolated component of the star graph,
no size-three ideal edge containing $\eta(g)$ is reductive.
For example if $\alpha\cup\gamma(1)$ and $\beta\cup\gamma(g)$ are not reductive,
one good hexagon has corners
\[
  \alpha,\ \delta(1),\ \beta,\ \gamma(g),\ \eta(g)\text{ and } \gamma(1).
\]

Now suppose no size-two ideal edge is reductive.
Again we will distinguish between the case $|u| < |v|$
and the case of equality.
In the former case, again we have that not both of
$\alpha\cup\gamma(1)$ and $\alpha\cup\gamma(g)$ are reductive
and similarly for $\beta\cup\gamma(1)$ and $\beta\cup\gamma(g)$.
If only one of these four ideal edges is reductive,
we have a good rectangle.
If $\alpha\cup\gamma(1)$ and $\beta\cup\gamma(1)$ are both reductive,
this would contradict \Cref{tworeductive} as well.
Suppose, then, that $\alpha\cup\gamma(1)$ and $\beta\cup\gamma(g)$ 
are both reductive.
Then $|v| > |u| \ge \{u,g^{-1}.u\}\cdot\{v,\bar v\} = |\{v,\bar v\}|$,
so $\{v\}\cdot\{\bar v\} > \frac{1}{2}|v|$.
From this it follows that $\delta(1)\cup\eta(1)$,
$\delta(g)\cup\eta(1)$, $\delta(1)\cup\eta(g)$ and $\delta(g)\cup\eta(g)$
are all not reductive.

Observe that at most two of the size-three ideal edges
$\alpha\cup\delta(1)$, $\beta\cup\delta(g)$
$\gamma(1)\eta(1)$ and $\gamma(g)\cup\eta(g)$
can be reductive by \Cref{whenreductive}.
The same holds true for $\alpha\cup\delta(g)$,
$\beta\cup\delta(1)$, $\gamma(g)\cup\eta(1)$ and $\gamma(1)\cup\eta(g)$.
If at most one of the four size-three ideal edges
$\alpha\cup\delta(1)$, $\beta\cup\delta(1)$,
$\alpha\cup\delta(g)$ and $\beta\cup\delta(g)$
is reductive,
we have a good rectangle.
If there are two, say $\alpha\cup\delta(1)$ and $\beta\cup\delta(1)$,
we again have a good rectangle by including $\delta(g)$ and either $\gamma$
ideal edge as corners.
A similar statement holds if instead 
$\alpha\cup\delta(g)$ and $\beta\cup\delta(g)$ are reductive.
If $\alpha\cup\delta(1)$ and $\beta\cup\delta(g)$ are reductive
but $\alpha\cup\delta(g)$ and $\beta\cup\delta(1)$ are not,
then $\gamma(g)\cup\eta(g)$ and $\gamma(1)\cup\eta(1)$ are not reductive
and one good hexagon from $\alpha$ to $\beta$ has corners
\[
  \alpha,\ \gamma(g),\ \beta,\ \gamma(1),\ \eta(1)\text{ and } \delta(g).
\]
A similar statement holds if the two reductive ideal edges 
containing a $\delta$ ideal edge
are $\alpha\cup\delta(g)$ and $\beta\cup\delta(1)$.

If three ideal edges containing a $\delta$ are reductive,
say $\alpha\cup\delta(1)$, $\alpha\cup\delta(g)$ and $\beta\cup\delta(g)$,
then $\{v,g.v\}\cdot\{u,\bar u\} = |\{u,\bar u\}|$,
while 
$\{v\}\cdot\{g.\bar u,u\} = \{v\}\cdot\{u\} \ge \frac{1}{2}|\{g.\bar u,u\}| = |u|$.
But this implies that $|\beta\cup\gamma(g)| = |u| + |\{g.v,g.\bar v\}|$,
contradicting the assumption that this ideal edge is reductive.

Finally suppose that $|u| = |v|$.
Then at most two ideal edges containing any given size-two ideal edge
may be reductive by \Cref{tworeductive}.
Consider, for example, the $\gamma$ and $\delta$ ideal edges.
At most two of these together with $\alpha$ may be reductive
and similarly for $\beta$.
If for some $\gamma$ or $\delta$ neither size-three ideal edge is reductive,
we have a good rectangle from $\alpha$ to $\beta$,
because for at least one of the remaining $\gamma$ or $\delta$ ideal edges,
at most one of the relevant size-three ideal edges is reductive. 

So suppose that for all $\gamma$ and $\delta$,
at least one of the size-three ideal edges containing it and either
$\alpha$ or $\beta$ is reductive.
But then notice that at least one of the size-three ideal edges
containing it and either $\eta(1)$ or $\eta(g)$ is not reductive.
In fact, we have a good hexagon unless
for each size-two ideal edge there are two others for which
the corresponding size-three ideal edge is reductive
and for, say, $\eta(1)$, either $\eta(1)$ and $\alpha$
have the same set of size-two ideal edges which yield 
(non)-reductive size-three ideal edges
or $\eta(1)$ and $\beta$ have this property
(and therefore similarly for $\eta(g)$.)

Suppose for definiteness that $\alpha\cup\gamma(1)$ and $\alpha\cup\delta(g)$
are reductive,
$\beta\cup\gamma(g)$ and $\beta\cup\delta(1)$ are reductive,
$\eta(1)\cup\gamma(1)$ and $\eta(1)\cup\delta(g)$ are reductive,
$\eta(g)\cup\gamma(g)$ and $\eta(g)\cup\delta(1)$ are reductive.
(Notice that swapping for instance 
$\delta(1)$ and $\delta(g)$ above is not possible
by \Cref{tworeductive}.)
We claim this is not possible.
Indeed, by \Cref{whenreductive},
we have that $\{u,g^{-1}.u\}\cdot\{v,\bar v\} = |\{v,\bar v\}|$,
while $0 = \{\bar u\}\cdot\{\bar v,g.v\} \ge \frac{1}{2}|\{\bar v,g.v\}|$,
which is impossible.
A similar argument works in all cases;
this completes the proof.

\section{Semistability at infinity}\label{semistabilitysection}
The goal of this section is to prove \Cref{semistability}.
Recall that we write $B_k$ for the ball of radius $k$.
We continue to let
\[
  N(k) = \max\{\|\tau\| : \tau \in C_k\},
\]
where $C_k$ denotes the set of reduced marked graphs of groups
(``briar patches'') whose stars have nonempty intersection with $B_k$.
To prove semistability,
we will show that when $L(F)$ is one ended,
for a proper ray
$\rho\colon \mathbb{R}_{\ge 0} \to L$,
with $\rho(k)$ a briar patch such that the sequence
$\|\rho(k)\| > N(k)$ is strictly increasing,
the inverse sequence
\[
  \pi_1(L - B_k, \rho(k))
\]
satisfies the Mittag-Leffler condition.

\begin{thm}\label{mainsemistability}
Suppose $L(F)$ is one ended.
If $n \ge N(k)$, then every loop in $L - B_{N(k)}$
based at $\rho(N(k))$
is homotopic to a loop in $L - B_n$ based at $\rho(n)$
by a homotopy that avoids $B_k$.
\end{thm}
When $L(F)$ is not one ended,
we will use our good algebraic understanding of $\out(F)$
to show that each end of $L(F)$ is semistable.

To prove \Cref{mainsemistability},
we need the following proposition.
\begin{prop}[cf.\ Proposition 4.1 of~\cite{Vogtmann}]\label{allpathsstandard}
  Every path in $L - B_k$ is homotopic to a standard path in $L - B_k$
  by a homotopy in $L - B_k$.
\end{prop}

To prove the proposition, we need a lemma.

\begin{lem}\label{forestlemma}
  Let $\mathcal{G}$ be a marked graph of groups
  representing a vertex of $L$
  with two maximal forests $F$ and $F'$.
  For each edge $e' \in F' - F$,
  there exists an edge $e \in F - F'$
  such that $F \cup\{e'\} - \{e\}$ is a maximal forest.
  Put another way, if $\mathcal{G}$ is not a briar patch,
  the collection of maximal forests
  in $\mathcal{G}$ forms a \emph{matroid.}
\end{lem}

Recall that a \emph{matroid} is simply a (nonempty) set,
some of whose finite subsets are called \emph{bases.}
The bases of a matroid satisfy an \emph{exchange condition}
similar to the one in the statement of the lemma.
Other examples include the set of spanning trees in a finite graph which is not a rose
or the set of bases of a vector space of finite positive dimension.

\begin{proof}
  Let $F$, $F'$ and $e'$ be as in the statement.
  Since $F$ was maximal,
  $F \cup \{e'\}$ fails to be a maximal forest,
  either because as a graph it has rank one
  or because it connects a pair of vertices $a$ and $b$ with nontrivial vertex group.
  Let $C$ be either the minimal subgraph supporting the (ordinary)
  fundamental group of $F \cup \{e'\}$
  in the former case or the unique geodesic from $a$ to $b$ in the latter.
  Now because $F'$ is a maximal forest,
  $C$ cannot be completely contained in $F'$,
  so any edge of $F - F'$ contained in $C$ will do as $e$.
\end{proof}

\begin{proof}[Proof of \Cref{allpathsstandard}.]
  If $\gamma$ is a marked graph of groups which lies outside of $B_k$,
  then every marked graph of groups it collapses onto is also outside of $B_k$
  by definition.
  Thus if we have a path $P = (\gamma_0,\gamma_1,\ldots,\gamma_\ell)$,
  by collapsing forests in alternating marked graphs of groups,
  we may assume that each $\gamma_k$ with $k$ even is a briar patch.
  If $\ell$ is odd, we may extend $P$ by some collapse of $\gamma_\ell$
  to a path of even length.
  If $\gamma_{k-1}$ and $\gamma_{k+1}$ are briar patches obtained in this way,
  they are obtained from $\gamma_k = (\mathcal{G},\sigma)$ by collapsing
  maximal forests $F$ and $F'$ in $\mathcal{G}$.

  By repeatedly applying \Cref{forestlemma},
  we see that we may interpolate from $F$ to $F'$
  through maximal forests by adding an edge of $F'$
  and removing one of $F$ at each step.
  The path segment $(\gamma_{k-1},\gamma,\gamma_{k+1})$
  is homotopic to the standard path segment
  \[
    (\gamma_{k-1} = \rho_0,\delta_1,\rho_1,\ldots,\delta_s,\rho_s=\gamma_{k+1})
  \]
  where $\rho_i$ is obtained from $\gamma_k$ by collapsing the $i$th maximal forest $F_i$
  in the interpolation
  and $\delta_i$ is obtained by collapsing $F_{i-1}\cap F_i$.
\end{proof}

With the proposition in hand, we are ready to prove the theorem.

\begin{proof}[Proof of \Cref{mainsemistability}.]
  Let $P$ be a loop in $L - B_{N(k)}$ based at $\rho(N(k))$ and let $n \ge N(k)$.
  By concatenating at both ends with an arbitrary path in $L - B_{N(k)}$
  from $\rho(N(k))$ to $\rho(n)$,
  we obtain a loop $P'$ based at $\rho(n)$.
  (Such a path exists by \Cref{mainconnectivity}.)
  By \Cref{allpathsstandard},
  $P'$ is homotopic relative to its basepoint to a standard path
  by a homotopy in $L - B_{N(k)}$.
  By the proof of \Cref{mainconnectivity},
  since $\|\rho(n)\| > N(n)$,
  we may push $P'$ outside of $B_n$ by a homotopy.
  This homotopy strictly increases the norm of briar patches along it.
  Since $P'$ began in $L - B_{N(k)}$
  and $N(k)$ is chosen so that the star of any briar patch in $L - B_{N(k)}$
  has empty intersection with the ball $B_k$,
  this homotopy avoids $B_k$.
\end{proof}

To complete the proof of \Cref{semistability},
we need only consider the cases where $\out(F)$ has infinitely many ends.
In the case where $\dim L = 1$,
we see that $\out(F)$ is virtually free.
Since free groups are semistable at each end
(consider the Mittag-Leffler condition in a tree)
and by~\cite[Proposition 16.5.3]{Geoghegan}
semistability passes to and is inherited from finite-index subgroups,
we see that $\out(F)$ is semistable at each end in this case.

Thus it remains to consider the case of $F = A_1 * A_2 * \mathbb{Z}$.
Observe that $A_1$ and $A_2$ are a complete set of representatives of
the conjugacy classes of maximal finite subgroups
of $F$, so these conjugacy classes are permuted by $\out(F)$.
There is a normal subgroup of index at most two in $\out(F)$
comprising the pointwise stabilizer of these conjugacy classes.
Call this subgroup $G$.

If $A$ is a group, the group $\hol(A) = A \rtimes \aut(A)$ always admits an automorphism $\theta$
(which is an involution when it is nontrivial)
sending an element $(a,\psi)$, with $a \in A$ and $\psi \in \aut(A)$,
to the element $(a^{-1}, \iota_{a^{-1}}\psi)$, where $\iota_{a^{-1}}$ is the inner automorphism
$x \mapsto axa^{-1}$.

Fix a generator $t$ of a $\mathbb{Z}$ free factor.
The group $G$ contains the finite group $(\hol(A_1) \times \hol(A_2))\rtimes C_2$,
where the cyclic group of order two acts by $\theta$ on each direct factor.
Explicitly, the subgroup $A_1$ acts by left multiplication on $t$,
the subgroup $A_2$ acts by right multiplication on $t$,
the subgroup $\aut(A_1)$ acts on the $A_1$ free factor, $\aut(A_2)$ on the $A_2$ free factor,
and $C_2$ acts by the automorphism
\[
  \tilde{\theta} = \begin{cases}
    a_1 &\mapsto a_1 \\
    a_2 &\mapsto t a_2 t^{-1} \\
    t &\mapsto t^{-1}.
  \end{cases}
\]
One checks that the intersection of this group of \emph{automorphisms}
with the subgroup of inner automorphisms is trivial,
so it injects into $\out(F)$ and in fact into $G$.

We will also consider the infinite, one-ended subgroup of automorphisms
generated by $\hol(A_1)\times\hol(A_2)$ and the automorphism ${\tau}$
which simply inverts $t$.
This subgroup, which we will call $H$, also injects into $G$.
To see that $H$ is one-ended, observe that it has a map to the finite groups
$\aut(A_1\times A_2)$ and $C_2$,
defined by first using the universal property to induce an automorphism of
$A_1 \times A_2 \times \mathbb{Z}$ and then looking alternately
at the action on the torsion subgroup
or the action on $\mathbb{Z}$.
The intersection $H_0$ of the kernels, which has finite index in $H$,
comprises those automorphisms which (one checks) act trivially on the $A_1$ and $A_2$ free factors
and sends $t$ to an element $gth$, where $g$ and $h$ are elements of $A_1 * A_2$,
so $H_0$ is isomorphic to $(A_1 * A_2) \times (A_1 * A_2)$, a one-ended group
which is virtually the direct product of two free groups.

\begin{prop}
  The group $G$ is isomorphic to the amalgamated free product
  \[
    ((\hol(A_1)\times\hol(A_2))\rtimes C_2) *_{\hol(A_1)\times\hol(A_2)} H.
  \]
\end{prop}

It follows from the proposition that $G$ and hence $\out(F)$ has a finite-index subgroup
isomorphic to
\[
  (F_m \times F_m) * (F_m \times F_m),
\]
where $F_m$ is a finite-rank free group.
This latter group is a right-angled Artin group,
hence is semistable at each end by results of Mihalik~\cite{Mihalik}.

\begin{proof}
  Firstly, we claim that the subgroup
  $\hol(A_1)\times\hol(A_2)$ together with ${\tau}$ and $\tilde{\theta}$
  generate $G$.
  Indeed, it is not hard to see that $\aut(F)$ is generated by these elements
  together with the family of automorphisms
  \[
    \begin{cases}
      a_1 \in A_1 &\mapsto a_2 a_1 a_2^{-1} \\
      a_2 \in A_2 &\mapsto a_2 \\
      t &\mapsto t
    \end{cases}
    \quad
    \begin{cases}
      a_1 \in A_1 &\mapsto a_1 \\
      a_2 \in A_2 &\mapsto a_1a_2a_1^{-1} \\
      t &\mapsto t
    \end{cases}
    \quad
    \begin{cases}
      a_1 \in A_1 &\mapsto ta_1t^{-1} \\
      a_2 \in A_2 &\mapsto a_2 \\
      t &\mapsto t,
    \end{cases}
  \]
  as well as possibly some automorphism permuting $A_1$ and $A_2$ in the case they are isomorphic.
  (See for instance~\cite{Gilbert} for a finite presentation of $\aut(F)$.)
  Each of these first three differs from an automorphism in
  $\langle\hol(A_1)\times\hol(A_2),\tilde{\tau},\tilde{\theta}\rangle$
  by an inner automorphism,
  so the claim follows.
  To finish the proof, we will use a finite presentation for $\out(F)$,
  essentially due to Fouxe-Rabinovitch~\cite{Fouxe-Rabinovitch},
  following notation of Marchand~\cite{Marchand}.

  Let us fix notation.
  For $\phi_i$ an element of $\aut(A_i)$, we abuse notation and think of
  $\phi_i$ as an element of $G$ via its action on the fixed $A_i$ free factor of $F$.
  The subgroup of $G$ generated by the $\phi_i$ as $\phi_i$ varies in $\aut(A_i)$
  and as $i$ varies over $\{1,2\}$ is isomorphic to $\aut(A_1)\times\aut(A_2)$.
  Given $i \ne j$ and $\gamma_i \in A_i$, the \emph{partial conjugation} of $A_j$ by $\gamma_i$,
  is the automorphism $\alpha^{(\gamma_i)}_{ij}$ defined as sending
  $\gamma_j \in A_j$ to $\gamma_i^{-1}\gamma_j\gamma_i$ and fixing $A_i$ and $t$.
  We similarly define $\alpha_{ti}$ where we conjugate by $t$ instead of $\gamma_i$,
  so $\gamma_i \in A_i \mapsto t^{-1}\gamma_i t$.
  Finally we define the \emph{right} and \emph{left transvections} $\rho_i^{(\gamma_i)}$
  and $\lambda_i^{(\gamma_i)}$ as acting on $t$ by right (respectively left) multiplication
  by $\gamma_i \in A_i$.
  The group generated by the $\lambda_1^{(\gamma_1)}$, the $\rho_2^{(\gamma_2)}$
  and the $\Phi_i$ is the subgroup $\hol(A_1)\times\hol(A_2)$ described above.

  As a special case of~\cite[Theorem 5.2]{Marchand},
  we have that a defining system of relations for $G$
  with respect to the generating set
  \[
    \{\aut(A_1),\aut(A_2),\alpha_{ij}^{(\gamma_i)},\alpha_{ti},
    \lambda_{i}^{(\gamma_i)},\rho_i^{(\gamma_i)},\tau\}
  \]
  
  is given by the following.
  We follow the numbering in Marchand's paper, omitting relations that do not occur.
  Implicitly one includes relations specifying that, for instance,
  the subgroup generated by $\lambda_i^{(\gamma_i)}$ for fixed $i$
  but variable $\gamma_i \in A_i$ is isomorphic to $A_i$.
  (For the worried reader: think of automorphisms as acting on $F$ on the left.)
  \begin{enumerate}
  \item $\varphi_1\varphi_2 = \varphi_2\varphi_1$, so that
    $\langle\aut(A_1),\aut(A_2)\rangle\cong \aut(A_1)\times\aut(A_2)$.
  \item $\varphi_i\alpha_{tj} = \alpha_{tj}\varphi_i$ for $\varphi_i \in \aut(A_i)$
    and $i \ne j$.
  \item $\varphi_i \alpha_{ij}^{(\gamma_i)} = \alpha_{ij}^{(\varphi_i\gamma_i)}\varphi_i$
    for $\varphi_i \in \aut(A_i)$ and $\gamma_i \in A_i$.
  \item $\alpha_{t1}\alpha_{t2} = \alpha_{t2}\alpha_{t1}$.
  \item $\alpha_{ti}\alpha_{ij}^{(\gamma_i)}\alpha_{ti}^{-1}
    = \alpha_{tj}^{-1}\alpha_{ij}^{(\gamma_i)}\alpha_{tj}$
    for $\gamma_i \in A_i$.
    \setcounter{enumi}{11}
  \item $\lambda_i^{(\gamma_i)}\rho_j^{(\gamma_j)} = \rho_j^{(\gamma_j)}\lambda_i^{(\gamma_i)}$
    for $\gamma_i \in A_i$ and $\gamma_j \in A_j$, with $i$ and $j$ not necessarily distinct.
    Also $\tau\lambda_i^{(\gamma_i)} = \rho_i^{(\gamma_i^{-1})}\tau$.
  \item $\alpha_{ij}^{(\gamma_i)}\tau = \tau\alpha_{ij}^{(\gamma_i)}$.
  \item $\alpha_{ti}\tau = \tau\alpha_{ti}^{-1}$.
  \item $\rho_i^{(\gamma_i)}\varphi_j = \varphi_j\rho_i^{(\gamma_i)}$ for $\varphi_j \in \aut(A_j)$
    and $i \ne j$.
  \item $\rho_{i}^{(\varphi_i\gamma_i)}\varphi_i = \varphi_i\rho_i^{(\gamma_i)}$ for $\gamma_i \in A_i$,
    $\varphi_i \in \aut(A_i)$. This, together with some of the relations above implies that
    the subgroup generated by the $\lambda_1^{(\gamma_1)}$, $\rho_2^{(\gamma_2)}$ and the
    $\varphi_i$ is isomorphic to $\hol(A_1)\times\hol(A_2)$.
    \setcounter{enumi}{18}
  \item $\rho_i^{(\gamma_i)}\alpha_{ij}^{(\gamma_i)} = \alpha_{ij}^{(\gamma_i)}\rho_i^{(\gamma_i)}$.
  \item $\rho_i^{(\gamma_i)}\alpha_{tj} = \alpha_{tj}\rho^{(\gamma_i)}_i\alpha_{ij}^{(\gamma_i)}$.
  \item ${(\rho_i^{(\gamma_i)})}^{-1}\rho_j^{(\gamma_j)}\rho_i^{(\gamma_i)}\alpha_{ij}^{(\gamma_i)}
    = \alpha_{ij}^{(\gamma_i)}\rho_j^{(\gamma_j)}$.
  \item $\alpha_{ti}\rho_i^{(\gamma_i)} = \lambda_i^{(\gamma_i)}\alpha_{ti}\varphi_i$,
    where $\phi_i$ is the automorphism of $A_i$ defined by $x \mapsto \gamma_i^{-1}x\gamma_i$.
  \item $\tau\varphi_i = \varphi_i\tau$ for $\varphi_i \in \aut(A_i)$.
    \setcounter{enumi}{-1}
  \item $\varphi_i \alpha_{ij}^{(\gamma_i)} \rho_{i}^{(\gamma_i)}{(\lambda_{i}^{(\gamma_i)})}^{-1} = 1$,
    where $\varphi_i$ is the automorphism of $A_i$ defined by $x \mapsto \gamma_i^{-1}x\gamma_i$.
    Also $\alpha_{t1}\alpha_{t2} = 1$.
  \end{enumerate}
  We introduce a new generator $\tilde\theta$
  and the relation $\tilde\theta = \tau\alpha_{t2}$.
  As remarked already, the relations in items 0 and 12 allow us to discard
  generators down to the more minimal
  \[
    \{\hol(A_1)\times\hol(A_2),\tau,\tilde\theta\}.
  \]
  We have already seen that $\langle\hol(A_1)\times\hol(A_2),\tilde{\theta}\rangle$
  is isomorphic to the semidirect product
  $(\hol(A_1)\times\hol(A_2))\rtimes C_2$.
  Easy Tietze transformations yield a defining system of relations
  for $G$
  for the generating set
  \[
    \{\hol(A_1)\times\hol(A_2), \tilde{\theta}, \tau\}.
  \]
  Relevant to our purposes,
  we just need to check that every relation involving $\tilde{\theta}$
  takes place in $(\hol(A_1)\times\hol(A_2))\rtimes C_2$
  and every relation involving $\tau$ takes place in $H$.
  Thus we need only examine relations 2, 5, 14, 20, and 22.
  \begin{enumerate}
    \setcounter{enumi}{1}
  \item The relation $\varphi_1\alpha_{t2} = \alpha_{t2}\varphi_1$ becomes
    $\varphi_1\tau\tilde{\theta} = \tau\tilde{\theta}\varphi_1$.
    Because $\hol(A_1)\times\hol(A_2)$ is normal in
    $(\hol(A_1)\times\hol(A_2))\rtimes C_2$, right-multiplying by $\tilde{\theta}^{-1}$
    and then using a relation discussed before the proof shows that this relation
    involves only elements of $H$. A similar argument works for
    $\varphi_2\alpha_{t1} = \alpha_{t1}\varphi_2$, which becomes
    $\varphi_2\tilde{\theta}\tau = \tilde{\theta}\tau\varphi_2$.
    \setcounter{enumi}{4}
  \item This relation, observing that $\alpha_{t2} = \alpha_{t1}^{-1}$, is spurious.
    \setcounter{enumi}{13}
  \item This relation is also spurious: it says either that
    $\alpha_{t2}\tau\alpha_{t2} = \tau\tilde{\theta}\tau\tau\tilde{\theta} = \tau$,
    which follows because $\tau$ and $\tilde{\theta}$ have order two, or that
    $\alpha_{t1}\tau\alpha_{t1} = \alpha_{t2}^{-1}\tau\alpha_{t2}^{-1} = \tau$,
    which follows from the previous relation.
    \setcounter{enumi}{19}
  \item Because $\hol(A_1)\times\hol(A_2)$ is normal in the semidirect product,
    the case of this relation of the form
    $\rho_2^{(\gamma_2)}\alpha_{t1} = \alpha_{t1}\rho_2^{(\gamma_2)}\alpha_{21}^{(\gamma_2)}$
    may be written with only elements of $H$ by writing $\alpha_{t1} = \tilde{\theta}\tau$
    and left multiplying by $\tilde{\theta}^{-1}$.
    The same is true for the relation with the roles of $1$ and $2$ swapped,
    where instead we left multiply by $\tilde{\theta}^{-1}\tau^{-1}$ and observe
    that after conjugating $\rho_1^{(\gamma_1)}$ by $\tau$ we obtain an element of
    $\hol(A_1)\times\hol(A_2)$.
    \setcounter{enumi}{21}
  \item The same method of proof as in the previous case works:
    for example the relation $\alpha_{t1}\rho_1^{(\gamma_1)} = \lambda_1^{(\gamma_1)}\alpha_{t1}\varphi_1$
    may be rewritten as
    $\tau\rho_1^{(\gamma_1)} = \tilde{\theta}^{-1}\lambda_1^{(\gamma_1)}\tilde{\theta}\tau\varphi_1$,
    which takes place in $H$.
  \end{enumerate}
  Thus we have demonstrated that $G$ splits as the amalgamated free product in the statement.
\end{proof}
\section*{Appendix: a direct proof of \Cref{maintheorem}}
Suppose $F = A_1 * A_2 * \cdots * A_n * F_k$.
In view of results of Das~\cite{Das} and Vogtmann~\cite{Vogtmann},
one-endedness of $\out(F)$ is only new when both $n$ and $k$ are positive.
To illustrate a more elementary attack on the proof of \Cref{maintheorem},
we will give a sketch of an argument supposing that $n \ge 3$ and $k \ge 3$.
In view of our proof of \Cref{maintheorem}, it is clear that this argument is not optimal;
we leave its improvement as a final question.

As we saw in the previous section, there exists a finite-index subgroup $G$ of $\out(F)$
comprising those outer automorphisms preserving the conjugacy class of each $A_i$.

Fix a free basis $\{x_{n+1},\ldots,x_{n+k}\}$ for $F_k$.
Following notation from the previous section, we have for each nontrivial $\gamma_i \in A_i$,
each $j$ satisfying $n+1 \le j \le n + k$ and each $\ell$ satisfying $1 \le k \le n + k$,
\emph{partial conjugations} $\alpha_{i\ell}^{(\gamma_i)}$ and $\alpha_{j\ell}$
and \emph{right and left transvections,} which when $\ell \in \{n+1,\ldots,n+k\}$
we write as $\rho_{\ell j}$ and $\lambda_{\ell j}$, while when $i \in \{1,\ldots,n\}$,
we write as $\rho_{ij}^{(\gamma_i)}$ and $\lambda_{ij}^{(\gamma_i)}$,
as they depend on a choice of nontrivial $\gamma_i \in A_i$.

The group $G$ is generated by, one checks, these partial conjugations, transvections,
the groups $\aut(A_i)$ acting on $A_i$ alone,
permutations of the $x_\ell$ for $\ell \in \{n+1,\ldots,n+k\}$
and the \emph{inversions} $\tau_{\ell}$ which invert $x_\ell$ and fix the other factors.

Here is the argument we will use; it is due to the anonymous referee:
supposing that $\out(F)$ and hence $G$ had more than one end,
it would act on a tree $T$ with finite edge stabilizers.
Every one-ended subgroup of $G$ stabilizes a unique vertex of $T$.
In particular, every subgroup isomorphic to the direct product of two infinite groups stabilizes a unique vertex of $T$.
If $H \le G$ is an elliptic subgroup in an action of $G$ on a tree $T$
and $g \in G$ normalizes $H$, observe that $g$ preserves the fixed-point set of $H$.
In particular, if $H$ stabilizes a unique vertex of $T$, so does its normalizer.
If $H$ and $H'$ are elliptic subgroups in the action of $G$ on $T$,
their intersection $H \cap H'$ is also elliptic.
If edge stabilizers in $T$ are finite, then
if $H \cap H'$ is infinite, we conclude that $H$ and $H'$ fix the same point.

We find a family $\Gamma_1,\ldots,\Gamma_m$ of subgroups of $G$,
each isomorphic to the direct product of two infinite groups,
such that $\Gamma_i \cap \Gamma_{i+1}$ is infinite for each $i$,
and such that each element of the above generating set is contained in or normalizes some $\Gamma_i$.
It follows from the argument above that in any action of $G$ on a tree $T$ with finite edge stabilizers,
each $\Gamma_i$ and their normalizers stabilize a unique vertex of $T$
and moreover these vertices are a single vertex.
It follows that $G$ acts with global fixed point on $T$,
from which we conclude that $G$ is one ended.

For each two-element subset of $\{1,\ldots,n\}$, say $\{i,j\}$,
it is not hard to see that the
partial conjugations $\alpha_{i\ell}^{(\gamma_i)}$ and $\alpha_{j\ell}^{(\gamma_j)}$
as $\gamma_i$ and $\gamma_j$ vary in $A_i$ and $A_j$ respectively and $\ell$ varies in $\{1,\ldots,n\} - \{i,j\}$
together with the left and right transvections
$\rho_{i\ell}^{(\gamma_i)}$, $\rho_{j\ell}^{(\gamma_j)}$, $\lambda_{i\ell}^{(\gamma_i)}$ and $\lambda_{j\ell}^{(\gamma_j)}$
as $\ell$ varies in $\{n+1, \ldots, n+k\}$
generate a group $H_{ij}$ isomorphic to the direct product of $2k + n - 2$ copies of the infinite group $A_i * A_j$.
If we fix $\ell$ satisfying $n+1 \le \ell \le n + k$,
the group $H_\ell$ generated by left and right transvections on the other elements of our fixed free basis
with ``acting letter'' $x_\ell$ is isomorphic to a free abelian group of rank $2k - 2$.
By mixing and matching products of direct factors of the $H_{ij}$ and $H_\ell$ as $i$, $j$ and $\ell$ vary,
we obtain a family of subgroups $\Gamma_1,\ldots,\Gamma_m$,
each isomorphic to the direct product of two infinite groups,
satisfying that $\Gamma_a \cap \Gamma_{a+1}$ is infinite for each $a$ satisfying $1 \le a \le m - 1$,
and such that each partial conjugation or transvection contained in some $H_{ij}$ or $H_\ell$
is contained in some $\Gamma_a$.

It is clear from our construction that because $n \ge 3$ and $k \ge 2$, each partial conjugation or transvection
in the generating set for $G$ sketched above normalizes (in fact centralizes) some $\Gamma_a$.
In fact, the elements of $\aut(A_i)$ and the inversions also clearly normalize some $\Gamma_a$.
Since we assume that $k \ge 3$, the same is also true of any transposition of two elements of
$\{x_{n+1},\ldots,x_{n+k}\}$.
It follows that the subgroups $\Gamma_a$ allow us to conclude that $G$ and hence $\out(F)$ is one ended.

Let us remark that if we knew that each $\Gamma_a$ was furthermore \emph{undistorted} in $G$,
we could conclude a stronger statement,
namely that $G$ and hence $\out(F)$ is \emph{thick of order one} in the sense of Behrstock,
Dru\c{t}u and Mosher~\cite{BehrstockDrutuMosher}.
When $k = 0$ and $n \ge 4$, this is the main result of Das's paper~\cite{Das},
and for $k \ge 3$ and $n = 0$, this was proven in~\cite{BehrstockDrutuMosher}.

\begin{question}
  Let $F = A_1 * \cdots * A_n * F_k$,
  and suppose that $\out(F)$ is one ended.
  Is it also thick of order one?
\end{question}

\bibliographystyle{amsplain}
\bibliography{bib.bib}

\providecommand{\bysame}{\leavevmode\hbox to3em{\hrulefill}\thinspace}
\providecommand{\MR}{\relax\ifhmode\unskip\space\fi MR }
\providecommand{\MRhref}[2]{%
  \href{http://www.ams.org/mathscinet-getitem?mr=#1}{#2}
}
\providecommand{\href}[2]{#2}
\begin{thebibliography}{10}

\bibitem{Bass}
Hyman Bass, \emph{Covering theory for graphs of groups}, J. Pure Appl. Algebra \textbf{89} (1993), no.~1-2, 3--47. \MR{1239551}

\bibitem{BehrstockDrutuMosher}
Jason Behrstock, Cornelia Dru\c{t}u, and Lee Mosher, \emph{Thick metric spaces, relative hyperbolicity, and quasi-isometric rigidity}, Math. Ann. \textbf{344} (2009), no.~3, 543--595. \MR{2501302}

\bibitem{BestvinaFeighnDuality}
Mladen Bestvina and Mark Feighn, \emph{The topology at infinity of {${\rm Out}(F_n)$}}, Invent. Math. \textbf{140} (2000), no.~3, 651--692. \MR{1760754}

\bibitem{BieriEckmann}
Robert Bieri and Beno Eckmann, \emph{Groups with homological duality generalizing {P}oincar\'{e} duality}, Invent. Math. \textbf{20} (1973), 103--124. \MR{340449}

\bibitem{BradyMeier}
Noel Brady and John Meier, \emph{Connectivity at infinity for right angled {A}rtin groups}, Trans. Amer. Math. Soc. \textbf{353} (2001), no.~1, 117--132. \MR{1675166}

\bibitem{Clay}
Matt Clay, \emph{Deformation spaces of {$G$}-trees and automorphisms of {B}aumslag-{S}olitar groups}, Groups Geom. Dyn. \textbf{3} (2009), no.~1, 39--69. \MR{2466020}

\bibitem{CollinsZieschang}
Donald~J. Collins and Heiner Zieschang, \emph{Rescuing the {W}hitehead method for free products. {I}. {P}eak reduction}, Math. Z. \textbf{185} (1984), no.~4, 487--504. \MR{733769}

\bibitem{CullerVogtmann}
Marc Culler and Karen Vogtmann, \emph{Moduli of graphs and automorphisms of free groups}, Invent. Math. \textbf{84} (1986), no.~1, 91--119. \MR{830040}

\bibitem{Das}
Saikat {Das}, \emph{{Thickness of $\mathsf{Out}(A_1*...*A_n)$}}, Available at arXiv:1811.00435 [Math.GR], Nov 2018.

\bibitem{Fouxe-Rabinovitch}
D.~I. Fouxe-Rabinovitch, \emph{\"{U}ber die {A}utomorphismengruppen der freien {P}rodukte. {II}}, Rec. Math. [Mat. Sbornik] N. S. \textbf{9 (51)} (1941), 183--220. \MR{0004625}

\bibitem{Geoghegan}
Ross Geoghegan, \emph{Topological methods in group theory}, Graduate Texts in Mathematics, vol. 243, Springer, New York, 2008. \MR{2365352}

\bibitem{Gilbert}
N.~D. Gilbert, \emph{Presentations of the automorphism group of a free product}, Proc. London Math. Soc. (3) \textbf{54} (1987), no.~1, 115--140. \MR{872253}

\bibitem{GuirardelLevittDeformation}
Vincent Guirardel and Gilbert Levitt, \emph{Deformation spaces of trees}, Groups Geom. Dyn. \textbf{1} (2007), no.~2, 135--181. \MR{2319455}

\bibitem{KrsticVogtmann}
Sava Krsti\'{c} and Karen Vogtmann, \emph{Equivariant outer space and automorphisms of free-by-finite groups}, Comment. Math. Helv. \textbf{68} (1993), no.~2, 216--262. \MR{1214230}

\bibitem{MyselfCAT0}
{Rylee Alanza} {Lyman}, \emph{{A family of CAT(0) outer automorphism groups of free products}}, Available at arXiv:2209.04711 [math.GR], September 2022.

\bibitem{Myself}
\bysame, \emph{{Lipschitz metric isometries between Outer Spaces of virtually free groups}}, Available at arXiv:2203.09008 [math.GR], March 2022.

\bibitem{MyTrainTracks}
\bysame, \emph{{Train track maps on graphs of groups}}, Available at arXiv:2102.02848 [math.GR], March 2022.

\bibitem{Marchand}
Alexis Marchand, \emph{Free representations of outer automorphism groups of free products via characteristic abelian coverings}, J. Group Theory \textbf{26} (2023), no.~2, 399--420. \MR{4554492}

\bibitem{Mihalik}
Michael~L. Mihalik, \emph{Semistability of {A}rtin and {C}oxeter groups}, J. Pure Appl. Algebra \textbf{111} (1996), no.~1-3, 205--211. \MR{1394352}

\bibitem{ScottWall}
Peter Scott and Terry Wall, \emph{Topological methods in group theory}, Homological group theory ({P}roc. {S}ympos., {D}urham, 1977), London Math. Soc. Lecture Note Ser., vol.~36, Cambridge Univ. Press, Cambridge-New York, 1979, pp.~137--203. \MR{564422}

\bibitem{Trees}
Jean-Pierre Serre, \emph{Trees}, Springer Monographs in Mathematics, Springer-Verlag, Berlin, 2003, Translated from the French original by John Stillwell, Corrected 2nd printing of the 1980 English translation. \MR{1954121}

\bibitem{Vogtmann}
Karen Vogtmann, \emph{End invariants of the group of outer automorphisms of a free group}, Topology \textbf{34} (1995), no.~3, 533--545. \MR{1341807}

\end{thebibliography}
\end{document}